\documentclass[12pt, hyphens]{article}
\usepackage[a4paper, 
            top=1.0in, 
            bottom=1.0in, 
            left=1.0in, 
            right=1.0in,
            hmarginratio=1:1,
            bindingoffset=0.0in
]
{geometry}
\usepackage[utf8]{inputenc}
\usepackage[T1]{fontenc}
\usepackage{lmodern}
\usepackage[english]{babel}
\usepackage{indentfirst}
\usepackage{amsmath,amssymb,mathtools}
\usepackage{mathrsfs} 
\usepackage{microtype}
\usepackage{bbm}
\usepackage{xparse}
\usepackage{etoolbox}
\usepackage{xcolor}
\usepackage[numbers, sort&compress]{natbib}
\usepackage[unicode=true, pdfnewwindow]{hyperref}
\hypersetup{pdfencoding=auto, psdextra, colorlinks=false, pdfborder={0 0 1}, linktoc=page}
\usepackage[numbered]{bookmark}
\usepackage{amsthm}
\theoremstyle{plain}
\newtheorem{theorem}{Theorem}[section]
\newtheorem{lemma}[theorem]{Lemma}
\newtheorem{proposition}[theorem]{Proposition}
\newtheorem{fact}[theorem]{Fact}

\theoremstyle{definition}
\newtheorem{definition}[theorem]{Definition}
\newtheorem{remark}[theorem]{Remark}
\newtheorem{example}[theorem]{Example}
\newtheorem{assumption}[theorem]{Assumption}

\numberwithin{equation}{section}

\makeatletter
\if@titlepage%
\else
    \def\abstract@style{heading} 
    \def\abstract@size{\small}
    
    \newcommand{\setabstractstyle}[1]{%
      \def\abstract@style{#1}%
    }
    
    \newcommand{\setabstractsize}[1]{%
      \def\abstract@size{#1}%
    }
    
    \let\orig@abstract\abstract
    \let\endorig@abstract\endabstract
    
    \renewenvironment{abstract}{%
      \if@twocolumn
        \orig@abstract%
      \else
        \abstract@size%
        \edef\temp@a{inline}%
        \ifx\abstract@style\temp@a%
          \quotation
          \noindent\textbf{\abstractname.}\ %
          \ignorespaces
        \else
          \begin{center}%
          {\bfseries\abstractname\vspace{-.5em}\vspace{\z@}}%
          \end{center}%
          \quotation
          \ignorespaces
        \fi
      \fi
    }{%
      \if@twocolumn
        \endorig@abstract
      \else
        \endquotation 
      \fi
    }
\fi
\makeatother
\setabstractstyle{inline} 
\setabstractsize{\normalsize} 
\renewcommand{\proofname}{Proof}%
\makeatletter
\renewenvironment{proof}[1][\proofname]{%
  \par\pushQED{\qed}\normalfont
  \topsep6\p@\@plus6\p@\relax 
  \trivlist
  \item[\hskip\labelsep\bfseries #1\@addpunct{.}]\ignorespaces
}{%
  \popQED\endtrivlist\@endpefalse
}
\makeatother
\usepackage{tocloft}
\makeatletter
\newif\ifinappendixbookmark
\inappendixbookmarkfalse

\let\oldappendix\appendix
\renewcommand{\appendix}{%
  \oldappendix
  \inappendixbookmarktrue 
}
\makeatother

\newif\ifdisableappendixbookmark
\disableappendixbookmarkfalse

\makeatletter
\newcommand*{\AddAppendixPrefixInBookmarks}{%
  \ifinappendixbookmark
    \ifnum\bookmarkget{level}=1
      \ifdisableappendixbookmark
      \else
        \preto\bookmark@text{Appendix\space}%
      \fi
    \fi
  \fi
}
\bookmarksetup{
  addtohook=\AddAppendixPrefixInBookmarks,
}
\makeatother
\makeatletter
\renewcommand\Hy@numberline[1]{#1.\space}%
\makeatother

\makeatletter
\newif\if@appendixsec
\@appendixsecfalse
\apptocmd{\appendix}{%
  \@appendixsectrue
}{}{}

\renewcommand{\@seccntformat}[1]{%
  \if@appendixsec
    \ifstrequal{#1}{section}{%
      Appendix~\thesection.\hskip.5em
    }{%
      \csname the#1\endcsname.\hskip.5em
    }%
  \else
    \csname the#1\endcsname.\hskip.5em
  \fi
}
\makeatother
\makeatletter
\newif\ifinappendixtoc
\inappendixtocfalse

\let\oldappendixtoc\appendix
\renewcommand{\appendix}{%
  \oldappendixtoc
  \addtocontents{toc}{\protect\inappendixtoctrue}%
}
\makeatother

\makeatletter
\let\oldnumberline\numberline
\renewcommand{\numberline}[1]{%
  \ifinappendixtoc
    \oldnumberline{\vphantom{[]}#1.\hskip.5em}
  \else
    \oldnumberline{#1.\vphantom{[]}}%
  \fi
}
\makeatother

\makeatletter
\let\oldcontentsline\contentsline

\renewcommand{\contentsline}[4]{%
  \ifinappendixtoc
    \ifstrequal{#1}{section}{%
      \oldcontentsline{#1}{\vphantom{[]}Appendix~#2}{\vphantom{[]}#3}{#4}%
    }{%
      \oldcontentsline{#1}{\vphantom{[]}#2}{\vphantom{[]}#3}{#4}%
    }%
  \else
    \oldcontentsline{#1}{\vphantom{[]}#2}{\vphantom{[]}#3}{#4}%
  \fi
}
\makeatother
\makeatletter
\AtBeginDocument{%
    \pretocmd{\NAT@hyper@}{\vphantom{[]}}{}{}
    \def\def@NAT@last@yr#1{%
     \protected@edef\NAT@last@yr{%
      #1%
      \noexpand\mbox{%
       \noexpand\hyper@natlinkstart{\@citeb\@extra@b@citeb}%
       {\noexpand\citenumfont{\NAT@num\vphantom{[]}}}%
       \noexpand\hyper@natlinkend
      }%
     }%
    }%
}
\makeatother

\let\oldbibliographystyle\bibliographystyle
\renewcommand{\bibliographystyle}[1]{%
    \addtocontents{toc}{\protect\inappendixtocfalse}
    \disableappendixbookmarktrue
    \oldbibliographystyle{#1}%
}

\AtEndDocument{
\addcontentsline{toc}{section}{\refname}
}

\let\LTXlabel\label
\makeatletter
\renewcommand{\footnote}[2][\empty]{%
  \nolinebreak%
  \refstepcounter{footnote}
  \global\edef\sfootnote@arabic{\arabic{footnote}}%
  \global\protected@edef\sfootnote@the{\thefootnote}%
  \ltx@ifpackageloaded{hyperref}{%
    \ifHy@hyperfootnotes%
      \refstepcounter{Hfootnote}
      \global\let\Hy@saved@currentHref\@currentHref%
      \hyper@makecurrent{Hfootnote}%
      \global\let\Hy@footnote@currentHref\@currentHref%
      \global\let\@currentHref\Hy@saved@currentHref%
    \fi%
  }{}%
  \xdef\sfootnote@opt{#1}%
  \ifx\sfootnote@opt\empty%
    \footnotetext{\LTXlabel{fnr:\sfootnote@arabic}#2}%
  \else%
    \ltx@ifpackageloaded{hyperref}{%
      \footnotetext[#1]{\phantomsection\LTXlabel{fnr:\sfootnote@arabic}#2\vphantom{Xg}}%
    }{%
      \footnotetext[#1]{\LTXlabel{fnr:\sfootnote@arabic}#2}%
    }%
  \fi%
  \ltx@ifpackageloaded{hyperref}{%
    \ifHy@hyperfootnotes%
      \hbox{\@textsuperscript{\,\normalfont\ref{fnr:\sfootnote@arabic}\vphantom{Xg}}}
    \else%
      \hbox{\@textsuperscript{\normalfont\ref*{fnr:\sfootnote@arabic}}}%
    \fi%
  }{%
    \hbox{\@textsuperscript{\normalfont\ref{fnr:\sfootnote@arabic}}}%
  }%
  \;%
  \ignorespaces%
}
\makeatother

\makeatletter
\newif\if@maketitle@nohyper
\let\orig@maketitle\maketitle
\renewcommand{\maketitle}{%
  \@maketitle@nohypertrue
  \orig@maketitle
  \@maketitle@nohyperfalse
}
\makeatother

\makeatletter
\renewcommand{\footnotemark}[1][]{%
  \leavevmode
  \ifhmode\edef\@x@sf{\the\spacefactor}\nobreak\fi
  \def\@tempa{#1}%
  \ifx\@tempa\@empty
    \stepcounter{footnote}%
    \protected@xdef\@thefnmark{\thefootnote}%
  \else
    \begingroup
      \c@footnote #1\relax
      \unrestored@protected@xdef\@thefnmark{\thefootnote}%
    \endgroup
  \fi
  \ltx@ifpackageloaded{hyperref}{%
  \ifHy@hyperfootnotes
    \if@maketitle@nohyper
    \else
      \refstepcounter{Hfootnote}%
      \global\let\Hy@saved@currentHref\@currentHref
      \hyper@makecurrent{Hfootnote}%
      \global\let\Hy@footnote@currentHref\@currentHref
      \global\let\@currentHref\Hy@saved@currentHref
    \fi
  \fi
  }{}%
  \textsuperscript{%
    \normalfont
    \ltx@ifpackageloaded{hyperref}{%
      \ifHy@hyperfootnotes
        \if@maketitle@nohyper
          \@thefnmark
        \else
          \,\hyper@linkstart{link}{\Hy@footnote@currentHref}%
          \@thefnmark\vphantom{Xg}%
          \hyper@linkend
        \fi
      \else
        \@thefnmark
      \fi
    }{%
       \@thefnmark
    }%
  }%
  \ifhmode\spacefactor\@x@sf\fi
  \;\ignorespaces
}
\makeatother
\newcommand{\myref}[1]{%
  \mbox{\ref{#1}\vphantom{()}}%
}
\makeatletter
\DeclareRobustCommand{\myhyperref}[2]{%
  \mbox{%
    \hyperref[#2]{%
     \ref*{#1}\;\textup{\tagform@{\ref*{#2}}}\,%
    }\vphantom{()}%
  }%
}
\makeatletter
\DeclareRobustCommand{\myeqhyperrefrange}[4]{%
  \mbox{%
    \hyperref[#2]{%
      \ref*{#1}\;\textup{\tagform@{\ref*{#3}}}\,--\,\textup{\tagform@{\ref*{#4}}}\,%
    }\vphantom{()}%
  }%
}
\makeatother

\newcommand{\myEqref}[1]{\mbox{\hyperref[#1]{(\ref*{#1})}\vphantom{()}}}
\makeatletter
\NewDocumentCommand{\runinsectionstar}{o m}{%
  \IfNoValueTF{#1}{%
    \def\runin@size{\Large}%
  }{%
    \def\runin@size{#1}%
    \if\relax\detokenize{#1}\relax
      \def\runin@size{\Large}%
    \fi
  }%
  \addcontentsline{toc}{section}{#2}%
  \@startsection{section}%
    {1}%
    {\z@}%
    {-3.5ex \@plus -1ex \@minus -.2ex}%
    {-1em}%
    {\normalfont\runin@size\bfseries}*%
    {#2.}%
}
\makeatother
\makeatletter
\NewDocumentCommand{\runinsubsectionstar}{o m}{%
  \IfNoValueTF{#1}{%
    \def\runin@size{\large}%
  }{%
    \def\runin@size{#1}%
    \if\relax\detokenize{#1}\relax
      \def\runin@size{\large}%
    \fi
  }%
  \addcontentsline{toc}{subsection}{#2}%
  \@startsection{subsection}%
    {2}%
    {\z@}%
    {-3.25ex \@plus -1ex \@minus -.2ex}%
    {-1em}%
    {\normalfont\runin@size\bfseries}*%
    {#2.}%
}
\makeatother

\makeatletter
\NewDocumentCommand{\runinsubsection}{o m}{%
  \IfNoValueTF{#1}{%
    \def\runin@size{\large}%
  }{%
    \def\runin@size{#1}%
    \if\relax\detokenize{#1}\relax
      \def\runin@size{\large}%
    \fi
  }%
  \@startsection{subsection}%
    {2}%
    {\z@}%
    {-3.25ex \@plus -1ex \@minus -.2ex}%
    {-1em}%
    {\normalfont\runin@size\bfseries}%
    {#2.}%
}
\makeatother

\DeclareMathOperator*{\graph}{gph}
\DeclareMathOperator*{\minimize}{minimize}
\DeclareMathOperator*{\cl}{cl}
\DeclareMathOperator*{\interior}{int}
\DeclareMathOperator*{\conv}{conv\!}

\DeclareMathOperator*{\dist}{dist}

\newcommand{\isomorph}{%
  \mathrel{\lower0.2ex\hbox{$\cong$}}
}

\providecommand{\scalarp}[1]{\ensuremath \langle{#1}\rangle}

\providecommand{\indicator}{\ensuremath \mathbbm{1}}

\newcommand{\R}{\ensuremath \mathbb{R}}
\newcommand{\N}{\ensuremath \mathbb{N}}

\newcommand{\Rm}{\ensuremath \R^m}

\newcommand{\Rq}{\ensuremath \R^q}

\newcommand{\Ncal}{\ensuremath \mathcal{N}}

\usepackage{academicons}
\title{Piecewise smooth functions and conservative fields: calculus for nonsmooth nonconvex optimization beyond stratification}
\author{
Cheik Traor\'e\thanks{Toulouse School of Economics, Toulouse Capitole University, Toulouse, France \\ (\texttt{cheik.traore@tse-fr.eu}).}
}
\date{\today}
\begin{document}\maketitle
\begin{abstract}
In this paper, we show that the piecewise gradient associated with a representation of a continuous piecewise-$C^{p}$ function is a selection of a conservative field. Specifically, we prove that a set-valued map whose selections include the piecewise gradients, also called associated gradients, has the chain rule property along Lipschitz curves. As a consequence, continuous piecewise smooth functions are path differentiable and their Clarke subdifferentials satisfy the chain rule property. These results establish a connection between a representation-based calculus for nonsmooth automatic differentiation and the conservative-field framework. From an algorithmic perspective, we show that bounded iterates of the stochastic piecewise-gradient method converge to the set of conservative critical points. With an additional Lebesgue-null interface condition, a generic stepsize scaling, and a generic initialization, this result holds for the Clarke critical set. Finally, we demonstrate that, even with the interface condition, the graph of a Lipschitz continuous and piecewise-$C^{\infty}$ function need not admit a $C^1$ stratification satisfying Whitney's condition (a). Thus, the analysis developed here does not fall within the setting of Whitney stratifiable functions (e.g., semialgebraic functions).
\end{abstract}

[1.5ex]
\section{Introduction}

\runinsubsectionstar[\normalsize]{Background and objectives} Automatic differentiation has become a fundamental tool in modern machine learning. It powers the training of virtually all deep neural networks by enabling efficient gradient computation through backpropagation. While its empirical success is undeniable, a rigorous theoretical foundation for its use in nonsmooth settings has remained incomplete until recently.

At the heart of this issue lies the chain rule, the mathematical principle underlying automatic differentiation. For compositions of smooth elementary functions and operations, the chain rule ensures that derivatives can be propagated through the computation. Correctness guarantees for automatic differentiation are generally formulated under smoothness assumptions, but in nonsmooth settings, a more delicate analysis is required \cite{griewank2008evaluating}. Yet many machine learning models—notably neural networks using ReLU and other nonsmooth activation functions—are not everywhere differentiable. At nonsmooth points, automatic differentiation acts on a particular program representation, and its output need not belong to classical generalized derivatives such as the Clarke subdifferential \cite{bolte2020mathematical,bolte2021conservative}. This discrepancy has created a gap between the widespread practical use of nonsmooth automatic differentiation and its mathematical justification, motivating the development of new frameworks that restore its validity in those settings.

A major breakthrough in resolving this issue was achieved through the work of Bolte and Pauwels \cite{bolte2021conservative}. For locally Lipschitz functions, they introduced the notion of conservative fields, a class of set-valued mappings characterized by the validity of a chain rule along curves. This seemingly simple principle provides the necessary structure to extend differential calculus beyond smooth settings. Although the framework is general in scope, it shows in particular that, for semialgebraic functions\footnote{These are functions whose graphs are semialgebraic sets. A semialgebraic set is a finite union of sets that can be described by finitely many polynomial equalities and inequalities. Semialgebraic functions form a subclass of definable functions in an o-minimal structure \cite{van1996geometric}, which in turn form a subclass of functions whose graphs admit a Whitney stratification; see Definitions~\myref{def:strat}~and~\myref{def:whit}. By an abuse of language, we refer to the latter as Whitney stratifiable functions.}, the objects produced by automatic differentiation can be interpreted as selections of conservative fields \cite{bolte2020mathematical}. 
Indeed, suppressing some technical details, a finite numerical program built from semialgebraic operations and branching conditions, with a continuous input–output map, induces a continuous semialgebraic function $f\colon\Rm\to\R$ that admits a finite selection representation
$$
f(x)= f_{s(x)}(x) = \sum_{i=1}^{N}\indicator_{\mathcal{A}_i}(x)f_i(x),
$$
where $\indicator_{\mathcal{A}_i}(x) = 1$ if $x \in \mathcal{A}_i$ and $ \indicator_{\mathcal{A}_i}(x) = 0$ otherwise, the sets
$
\mathcal{A}_i\coloneq\{x\in\Rm:s(x)=i\}
$
are semialgebraic, $s\colon \Rm \to \{1,\ldots,N\}$ is a semialgebraic selector,
and the $f_i\colon \Rm \to \R$ are sufficiently regular semialgebraic path functions.
The generalized gradient---or selection derivative, in the authors’ terminology \cite{bolte2020mathematical}---associated with this representation is
\begin{equation}\label{eq:selectingradient}
\widehat{\nabla}^{s}f(x)
:=\nabla f_{s(x)}(x)
=\sum_{i=1}^{N}\indicator_{\mathcal{A}_i}(x)\nabla f_i(x).
\end{equation}
They proved in \cite{bolte2020mathematical} that automatic differentiation applied to the semialgebraic function $f$ outputs a selection derivative, which in turn is a selection of a conservative field for $f$. The conservative-field paradigm thereby provides a rigorous mathematical justification for automatic differentiation in a broad class of nonsmooth problems. Beyond this theoretical foundation, conservative fields retain many of the key properties of classical gradients, making them powerful tools for analyzing nonsmooth optimization algorithms and their convergence.

More recently, for Lipschitz continuous and piecewise-$C^1$ functions, Després \cite{despres2025associated} uncovered a similar piecewise construction inspired by the work of Murat and Trombetti \cite{Murat2003chainrule}. The latter authors established a chain rule for the composition of a Lipschitz function with a Lipschitz piecewise-$C^{1}$ mapping, revealing that a meaningful differential calculus can be recovered despite the lack of smoothness. Després' key idea is to exploit an explicit representation of a piecewise-$C^{1}$ function as above. Given a finite Borel partition on which the function coincides with smooth components, one defines a piecewise gradient, which Després termed an associated gradient, by selecting the gradient of the corresponding smooth piece at each point, similarly to \eqref{eq:selectingradient}. Thanks to the Murat--Trombetti result, this piecewise gradient satisfies the chain rule, providing a simple and constructive differential object that is particularly well-suited to analyzing nonsmooth automatic differentiation for piecewise-$C^1$ functions. Under various assumptions and definitions, other representation-based derivatives for piecewise-defined functions have also appeared in the literature; see \cite{Scholtes2012piecewise,Khan2015derivate,billingsley2022generalized,lee2020correctness, griewank2008evaluating} and the references therein.

Although both approaches, in their specific functional contexts, successfully address the pitfalls of nonsmooth automatic differentiation, the precise relationship between piecewise gradients and conservative fields had not been known. Specifically, it was not known whether the representation-dependent piecewise gradients introduced for piecewise-$C^{1}$ functions fit within the conservative-field framework, as the selection derivatives \eqref{eq:selectingradient} introduced for definable functions do \cite{bolte2020mathematical}. Moreover, unlike selection derivatives, the role of piecewise gradients (associated gradients) in optimization remains poorly understood. Initial work has focused primarily on the calculus properties that make them useful for nonsmooth automatic differentiation \cite{despres2025associated, despres2026limit}, while their effect on nonsmooth gradient-based training dynamics is explicitly identified as an open question in \cite[Section~5.2]{despres2025associated}. Likewise, no optimization application has been developed for the related intensional derivatives introduced in \cite{lee2020correctness} for piecewise-analytic functions defined over finite or countably infinite analytic partitions. Accordingly, for piecewise-$C^p$ functions, our paper has three objectives: (i) to establish the relationship between piecewise gradients and the optimization-friendly framework of conservative fields; (ii) to derive convergence guarantees for stochastic optimization methods driven by piecewise gradients; and (iii) to determine whether these results fall within the Whitney stratifiable setting used in previous work, such as \cite{bolte2020mathematical}. We accomplish the first two objectives and, for the third, show that the piecewise-$C^p$ setting is not subsumed by the Whitney stratifiable one.

\runinsubsectionstar[\normalsize]{Results and related work} First, for $p \in (\N\setminus\{0\}) \cup \{\infty\}$, we consider continuous piecewise-$C^{p}$ functions\footnote{In what follows, simple continuity will be assumed by default unless stated otherwise. We therefore refer to continuous piecewise-$C^p$ functions simply as piecewise-$C^p$ functions.}, with pieces defined on a partition of $\Rm$ made of finitely many Borel sets that are not necessarily smooth manifolds. For the first objective, at a given point, we introduce a set-valued mapping $D_f$ that collects all closure-active component gradients of these pieces; see~\eqref{eq:defconsvfield}. Naturally, the corresponding piecewise gradient is a selection of this set-valued map. We then prove that $D_f$ satisfies the chain rule along every Lipschitz curve and therefore defines a conservative field; see the main result, Theorem~\myref{theo:main}. As an immediate consequence, piecewise-$C^p$ functions are path differentiable---see \cite{bolte2021conservative} and similar definitions in \cite{valadier1989entrainement, borwein1998chain}---and the Clarke subdifferential exhibits the same chain rule property along curves. As far as we are aware, neither the chain rule for the Clarke subdifferential of a piecewise-$C^p$ function nor the resulting path differentiability has been explicitly stated in the literature, although both follow by combining existing results. In fact, by \cite[Proposition~2.3]{Sun2008lowner}, the Clarke subdifferential of a piecewise-$C^1$ function is a semismooth derivative \cite{mifflin1977semismooth, qi1993nonsmooth}. Proposition~6.2 of \cite{bolte2023subgradient} then implies that it is a conservative gradient and hence satisfies the chain rule; consequently, such functions are path differentiable. However, this argument does not extend to our main selection result about piecewise gradients, because a piecewise gradient need not belong to—and therefore need not be a selection of—the Clarke subdifferential; see Remark~\myref{rmk:partitionexample}. It does, however, always belong to the set-valued map $D_f$ that we construct and prove to be a conservative field in Theorem~\myref{theo:main}. We then prove in Proposition~\myref{prop:piecewise-cp-semismooth2} that the convex hull of the conservative field constructed here is also a semismooth generalized gradient in the sense of Norkin~\cite{ermol1998stochastic}\footnote{Norkin’s semismoothness is formulated relative to an admissible set-valued generalized derivative, not necessarily the Clarke subdifferential.}. However, general convex-valued conservative fields for piecewise-$C^p$ functions need not be semismooth. Example~\myref{ex:conservative-not-norkin} provides a counterexample.

In Section~\myref{sect:application}, we tackle the second objective by investigating the relevance of piecewise gradients for optimization. In fact, we use piecewise gradients to drive a stochastic first-order algorithm, termed the stochastic piecewise-gradient method. Assuming that $p\geq m$, bounded iterates generated by the method with vanishing stepsizes converge almost surely (a.s.) to a conservative critical set, i.e., the zeros of a conservative field. If, in addition, $p \geq 2$ and a Lebesgue-null interface condition~\eqref{eq:interface}\,\footnote{For example, the condition is satisfied for ReLU-based neural networks, which are piecewise smooth functions; see Remark~\myref{rmk:relu-interface}.} holds, we obtain the same conclusion for the Clarke critical set under the genericity conditions stated in Theorem~\myref{theo:clarkcrit}. 
The interface condition is sufficient to ensure that the objective is locally $C^2$ almost everywhere (a.e.), a key property for the proof technique. Without this condition, we show in Proposition~\myref{claim:counter} that a merely Lipschitz continuous and piecewise-$C^{\infty}$ function need not be locally $C^1$ on a set of positive measure. 

Because we prove that the convex hull of the conservative field $D_f$ is a (semismooth) generalized gradient, our convergence result is related to that in \cite{ermol1998stochastic}, where the Norkin generalized gradient is used to study the stochastic generalized-gradient method. The result in \cite{ermol1998stochastic} guarantees convergence only to conservative---or, in their terminology, generalized gradient---critical points. Other analyses either assume Whitney stratifiability directly or impose structural hypotheses for which Whitney stratifiable functions---in particular, semialgebraic and definable functions---are the known examples provided; see \cite{davis2020stochastic, bolte2023subgradient, lai2026convergence, xiao2024adam, ding2025stochastic, bianchi2022convergence, le2024nonsmooth}. 

Finally, in Theorem~\myref{theo:counter2} of Section~\myref{sect:counter}, we address the third objective and show that Lipschitz continuous and piecewise-$C^{\infty}$ functions with the Lebesgue-null interface condition~\eqref{eq:interface} are not necessarily Whitney stratifiable. Without the interface condition, we observe from Proposition~\myref{claim:counter} that Lipschitz continuous and piecewise-$C^{\infty}$ functions may fail to be locally $C^1$ on a set of positive measure, hence they may not be Whitney stratifiable; see Remark~\myref{rmk:counter}. Theorem~\myref{theo:counter2} strengthens this observation by exhibiting a Lipschitz continuous and piecewise-$C^{\infty}$ function that satisfies the Lebesgue-null interface condition~\eqref{eq:interface}---hence is locally $C^p$ a.e. for arbitrary $p \in \N$ by Lemma~\ref{lem:piecewiseCplocallyC2ae}---but is not Whitney stratifiable. From a theoretical viewpoint, a consequence of this result is that our analyses are not covered by existing analyses restricted to Whitney stratifiable functions, and that piecewise-$C^p$ functions possess sufficient structure---beyond Whitney stratification---to support a calculus compatible with nonsmooth nonconvex optimization. However, it is worth noting that many of the most practically important piecewise smooth functions are definable in some o-minimal structure, including, in particular, the ubiquitous ReLU-based neural networks. Therefore, from a practical viewpoint, our result offers a complementary perspective and an alternative geometric framework to Whitney stratification that can be exploited in the analysis of optimization methods for such functions.

\section{Preliminaries}
From an optimization perspective, we restrict our analysis to real-valued functions. Nevertheless, all our results concerning the link between piecewise derivatives and conservative fields extend seamlessly to vector-valued functions, in which case conservative fields are replaced by conservative Jacobians; see \cite[Section 3.3]{bolte2021conservative}. 

Throughout this document, the reference measure is the Lebesgue measure, and explicit mention of it will be omitted whenever there is no ambiguity.
\subsection{Piecewise smooth functions and piecewise gradients}
Let $m, {J} \in \N \setminus \{0\}$. A finite \emph{Borel partition} $\{P^{j}\}_{j=1}^{{J}}$ of $\Rm$ is defined as follows:
$
    \displaystyle \Rm = \bigsqcup_{j=1}^{{J}} P^{j},
$
where each $P^{j}$ is a nonempty Borel subset of $\Rm$, and
$P^{\ell} \cap P^i = \varnothing$ for all $\ell,i \in \{1,\ldots,{J}\}$ with $\ell \neq i$. 

We call the sets in the partition its \emph{parts}.

\begin{definition}\label{def:piecewisec1}
Let $p\in \N \cup \{\infty\}$ and let $f\colon \Rm\to\R$ be a continuous function. We say that $f$ is \emph{piecewise $p$-times continuously differentiable}, or \emph{piecewise-$C^{p}$}, if there exist a finite Borel partition
$
\Rm=\bigsqcup_{j=1}^{{J}} P^{j},
$
and $C^{p}$ functions $f^1,\ldots,f^{{J}}\colon \Rm\to\R$ such that
\begin{equation}\label{eq:representation}
f(x)=\sum_{j=1}^{{J}}\indicator_{P^{j}}(x)\,f^{j}(x),
\quad \forall x\in\Rm.
\end{equation}
The pair
$
\left(\{P^j\}_{j=1}^{J},\{f^j\}_{j=1}^{J}\right)
$
is called a \emph{representation} of $f$.
\end{definition}

In this document, we call any function piecewise smooth if it is piecewise-$C^p$ for some $p\in \left(\N \setminus \{0\}\right) \cup \{\infty\}$. We also recall that all functions are assumed to be continuous. For convenience, we therefore use the terms ``continuous piecewise-$C^p$ functions'' and ``piecewise-$C^p$ functions'' interchangeably.

\begin{assumption}[Standing assumption]
    Unless stated otherwise, we assume that $p\in \left(\N \setminus \{0\}\right) \cup \{\infty\}$.
\end{assumption}

\begin{fact}
    Every continuous piecewise-$C^p$ function defined as in Definition~\myref{def:piecewisec1} is locally Lipschitz continuous; see \cite[Corollary~4.1.1]{Scholtes2012piecewise}.
\end{fact}

Despr\'es \cite{despres2025associated} presented associated gradients, i.e., piecewise gradients, as a tool for automatic differentiation. Combined with the result of Murat and Trombetti \cite{Murat2003chainrule}, these gradients enable a chain rule suitable for automatic differentiation.

\begin{definition}[Piecewise gradient]
Let $f\colon \Rm \to \R$ be piecewise-$C^{p}$, and let $\{P^{j}\}_{j=1}^{{J}}$ and $\{f^{j}\}_{j=1}^{{J}}$ denote a Borel partition and a collection of $C^{p}$ functions defining a representation of $f$ as in Definition~\myref{def:piecewisec1}. The \emph{piecewise gradient} or \emph{associated gradient} of $f$ with respect to this representation is defined by
$$
\widetilde{\nabla}f(x)
=
\sum_{j=1}^{J}
\indicator_{P^{j}}(x)\,
\nabla f^{j}(x),
\quad
\forall x\in\Rm.
$$
Note that the piecewise gradient generally depends on the chosen representation.
\end{definition}

\begin{lemma}\label{lem:vectorspace}
Let $f,g\colon\Rm\to\R$ be piecewise-$C^p$, with respective
representations
$$
\left(
\{P_f^j\}_{j=1}^{J_f},
\{f^j\}_{j=1}^{J_f}
\right)
\quad\text{and}\quad
\left(
\{P_g^{\ell}\}_{\ell=1}^{J_g},
\{g^{\ell}\}_{\ell=1}^{J_g}
\right).
$$
Let $\beta\in\R$. Then $\beta f+g$ is piecewise-$C^p$. In
particular, $f+g$ and $\beta f$ are piecewise-$C^p$. 

If the piecewise gradient of $\beta f+g$ is defined
using the representation based on the \emph{common refinement} of the partitions $\{P_f^j\}_{j=1}^{J_f}$ and $\{P_g^{\ell}\}_{\ell=1}^{J_g}$, while the piecewise gradients of $f$ and
$g$ are defined using their original representations, then
$$
\widetilde{\nabla}(\beta f+g)
=
\beta\,\widetilde{\nabla}f
+
\widetilde{\nabla}g.
$$
\end{lemma}
\begin{remark}
After discarding empty sets, the family
$$
\mathcal{P}
\coloneq
\left\{
P_f^j\cap P_g^{\ell}:
j\in\{1,\ldots,J_f\},\
\ell\in\{1,\ldots,J_g\}
\right\} \setminus \{\varnothing\}
$$
is a finite Borel partition of $\Rm$, called the
\emph{common refinement} of the two original partitions. For every
nonempty part
$
P^{j,\ell}\coloneq P_f^j\cap P_g^{\ell},
$
define
$
h^{j,\ell}\coloneq \beta f^j+g^{\ell}.
$
Then the pairs $(P^{j,\ell},h^{j,\ell})$ define a representation of
$\beta f+g$.
\end{remark}
\begin{proof}[Proof of Lemma~\myref{lem:vectorspace}]
Since the sets $P_f^j$ and $P_g^{\ell}$ are Borel, every intersection
$P_f^j\cap P_g^{\ell}$ is Borel. Moreover, because the two original
families are partitions of $\Rm$, the nonempty intersections are
pairwise disjoint and satisfy
$$
\Rm
=
\bigsqcup_{\substack{1\leq j\leq J_f\\1\leq\ell\leq J_g}}
\left(P_f^j\cap P_g^{\ell}\right).
$$
They therefore form a finite Borel partition of $\Rm$. For $x\in P^{j,\ell}$, the original representations give
$$
f(x)=f^j(x)
\quad\text{and}\quad
g(x)=g^{\ell}(x).
$$
Hence
$$
(\beta f+g)(x)
=
\beta f^j(x)+g^{\ell}(x)
=
h^{j,\ell}(x).
$$
Each $h^{j,\ell}$ is $C^p$, so this yields a piecewise-$C^p$
representation of $\beta f+g$.

Finally, for every $x\in\Rm$, let $j$ and $\ell$ be the unique indices
such that $x\in P_f^j\cap P_g^{\ell}$. By the definition of the
piecewise gradients,
$$
\begin{aligned}
\widetilde{\nabla}(\beta f+g)(x)
&=
\nabla h^{j,\ell}(x) =
\beta\,\nabla f^j(x)+\nabla g^{\ell}(x) =
\beta\,\widetilde{\nabla}f(x)
+\widetilde{\nabla}g(x).
\end{aligned}
$$
The result follows.
\end{proof}
\begin{remark}
    The class of piecewise-$C^p$ functions is closed under composition, i.e., if $f\colon \Rq \to \R $ and $ g\colon \Rm \to \Rq$ are piecewise-$C^p$, then $f \circ g$ is piecewise-$C^p$; see \cite[Section~4.1]{Scholtes2012piecewise}.
\end{remark} 

\subsection{Conservative set-valued fields}

A \emph{set-valued map} or \emph{multifunction} $D\colon \Rm \rightrightarrows \Rq$ is a mapping from $\Rm$ to the set of all subsets of $\Rq$. The \emph{graph} of $D$ is given by $\graph D\coloneq\{(x, y) \in \Rm \times \Rq : y \in D(x)\}$. $D$ is \emph{locally bounded} at $x \in \Rm$ if there exist a neighborhood $\Ncal$ of $x$ and $r > 0$ such that $\displaystyle \bigcup_{z \in \Ncal} D(z) \subset B_c(0, r)\coloneq\{y\in \Rq: \|y\|\leq r\}$. $D$ is \emph{graph-closed} if $\graph D$ is a closed subset of $\Rm \times \Rq$. Equivalently, $D$ is graph-closed if for all $(x_{\ell})_{\ell \in \N} \subset \Rm$ and all $(y_{\ell})_{\ell \in \N} \subset \Rq$ such that $x_{\ell} \xrightarrow[\ell \to \infty]{} x$, $y_{\ell} \xrightarrow[\ell \to \infty]{} y$ and $y_{\ell} \in D(x_{\ell})$ for every $\ell \in \N$, it follows that $y \in D(x)$.

\begin{definition}[Clarke subdifferential \cite{clarke1983nonsmooth}]\label{def:clarkesubdiff} Let $\varphi\colon \Rm \to \R$ be a locally Lipschitz function. By Rademacher's theorem, $\varphi$ is differentiable on a full-measure subset of $\Rm$, say $\Omega_{\varphi}$. Then, the \emph{Clarke subdifferential} of $\varphi$ is the set-valued map $\partial^{c} \varphi\colon \Rm \rightrightarrows \Rm$ defined as
\begin{align*}
    x \mapsto \conv\,\{u \in \Rm: \exists (x_{\ell})_{\ell \in \N} \subset \Omega_{\varphi}, x_{\ell} \xrightarrow[\ell \to \infty]{} x \text{ and }  \nabla \varphi(x_{\ell}) \xrightarrow[\ell \to \infty]{} u \}.
\end{align*}
 Here, $\conv S$ denotes the convex hull of the set $S$. The Clarke subdifferential of $\varphi$ has nonempty, convex and compact values, and is graph-closed and locally bounded. 
\end{definition}

Our analysis revolves around the notion of conservative fields introduced in the next definition. 
\begin{definition}[Conservative fields \cite{bolte2021conservative}]\label{def:conservative} Let $\varphi\colon \Rm \to \R$ be a locally Lipschitz function, and let $D_{\varphi}\colon \Rm \rightrightarrows \Rm$ be a locally bounded and graph-closed set-valued map with nonempty values. We say that $D_{\varphi}$ is a \emph{conservative field} for $\varphi$ if for any Lipschitz\footnote{In the original literature \cite{bolte2021conservative}, this definition is stated using absolutely continuous curves. This formulation is equivalent, since, without altering its role in the definition, every absolutely continuous curve can be arc-length reparameterized as a Lipschitz continuous curve \cite[Lemma 1.1.4]{AmbrosioGigliSavare2008}.} curve $v\colon[0, 1] \to \Rm$,
    \begin{equation*}
        \frac{d}{dt}\varphi(v(t)) = \scalarp{\dot{v}(t), u}  \quad \forall u \in D_{\varphi}(v(t)), 
    \end{equation*}
for almost all $t \in [0, 1]$. Here, $\dot{v}(t)$ denotes $\frac{d}{dt} v(t)$. 

A locally Lipschitz function admitting such a conservative field is called \emph{path differentiable}.
\end{definition}
We conclude this section by introducing a set-valued map $D_f\colon \Rm \rightrightarrows \Rm$ that we will show to be a conservative field for $f$ in Theorem~\myref{theo:main}. 

Let $f\colon \Rm \to \R$ be continuous and piecewise-$C^{p}$, and let $\{P^{j}\}_{j=1}^{{J}}$ and $\{f^{j}\}_{j=1}^{{J}}$ denote a Borel partition and a collection of $C^{p}$ functions defining a representation of $f$ as in Definition~\myref{def:piecewisec1}. 

Define
\begin{align}\label{eq:defconsvfield}
D_f(x) 
&\coloneq 
\big\{ \nabla f^{j}(x): x \in  \cl P^{j}, j \in \{1, \ldots, {J}\}\big\},
\end{align}
where $\cl P^{j}$ denotes the closure of $P^{j}$.
\begin{remark}\label{rmk:selection}
    The piecewise gradient $\widetilde{\nabla} f$ of $f$ is a selection of $D_f$, i.e.,
    \begin{equation*}
        \widetilde{\nabla} f(x) \in  D_f(x) \quad \forall x \in \Rm.
    \end{equation*}
    $D_f$ is not unique because it depends on the chosen representation of $f$.
\end{remark}
\begin{remark}\label{rmk:partitionexample}
The Clarke subdifferential can be strictly contained in the convex hull $\conv D_f$\footnote{$\conv D_f\colon \Rm \rightrightarrows \Rm$ is the set-valued map defined as $x \mapsto \conv D_f(x)$.} of $D_f$. For example, in one dimension, let $f = |\cdot|$ be the absolute value function $|\cdot|\colon \R \to [0, \infty)$ defined by $x \mapsto x$ for $x \geq 0$ and $x \mapsto -x$ otherwise. $\partial^{c}f(0) = [-1, 1]$. Using the partition $\{(-\infty,0), \{0\}, (0, \infty)\}$ and the functions $\{-x, 2x, x\}$, we get $D_f(0) = \{-1, 2, 1\}$ and its convex hull $\conv D_f(0) = [-1, 2]$. This example also shows that piecewise gradients are not necessarily selections of the Clarke subdifferential: $\widetilde{\nabla}f(0) = 2 \notin [-1,1] = \partial^{c}f(0)$.
\end{remark}

\subsection{Technical results}

We state a result of Stampacchia, which is a key tool in the proofs of Theorem~\myref{theo:murat} and the main result in the next section.

\begin{lemma}[Stampacchia~{\cite[Lemma A.4]{stampacchia2000repub}}]\label{lem:stampacchia}
Let $s,N\in\N\setminus\{0\}$, let $\Omega\subset\R^N$ be a bounded open set, and let $v\in W^{1,s}(\Omega)$\footnote{That is, $v\colon \Omega \to \R$ has weak (distributional) first-order partial derivatives $\frac{\partial v}{\partial x_i}$, with $|v|^s$ and $\left|\frac{\partial v}{\partial x_i}\right|^s$ Lebesgue integrable for each $i\in\{1,\ldots,N\}$.}. Then,
$$
\frac{\partial v}{\partial x_i}(x)=0
\quad\text{a.e. }\quad x\in\{z\in\Omega:v(z)=0\},
\qquad i=1,\ldots,N.
$$
In particular, the same conclusion holds for every absolutely continuous $v \colon [0,1]\to\R$.
\end{lemma}

We now present a straightforward adaptation of the Murat--Trombetti chain rule result from globally Lipschitz functions to locally Lipschitz ones.

\begin{theorem}[Murat--Trombetti \cite{Murat2003chainrule} adapted]\label{theo:murat}
Let $f\colon \Rm \to \R$ be continuous and piecewise-$C^{p}$, and let $\{P^{j}\}_{j=1}^{{J}}$ and $\{f^{j}\}_{j=1}^{{J}}$ denote a Borel partition and a collection of $C^{p}$ functions defining a representation of $f$ as in Definition~\myref{def:piecewisec1}. Let $v\colon [0,1] \to \Rm$ be a Lipschitz continuous curve. Then, for almost all $t \in [0,1]$, the following chain rule holds:
\begin{equation*}
    \frac{d}{dt} f(v(t)) = \langle \dot{v}(t), \widetilde{\nabla} f(v(t)) \rangle.
\end{equation*}    
\end{theorem}
\begin{proof} The proof follows the arguments in \cite{Murat2003chainrule,despres2025associated,despres2026limit}. It is decomposed into three steps.

\noindent\emph{First step.} By the piecewise-$C^{p}$ decomposition of $f$, the classical chain rule gives, for each $j \in \{1, \ldots, J\}$,
\begin{equation}\label{eq:standardchainrule}
 \frac{d}{dt} f^{j}(v(t))
=
\langle \dot{v}(t), {\nabla} f^{j}(v(t)) \rangle \quad \text{a.e.}\quad t \in [0,1].
\end{equation}
Let
$$
U^{j}
=
\{s\in[0,1]:v(s)\in P^{j}\}.
$$
Since
$
U^{j}=v^{-1}(P^{j})
$
and $P^{j}$ is a Borel set while $v$ is continuous,
$U^{j}$ is also Borel measurable.

\noindent\emph{Second step.}
Observe that
\begin{equation}\label{eq:equalityonpartitionfonction}
f(v(t))
=
f^{j}(v(t))
\quad \forall t \in U^{j}.
\end{equation}
Define
$
w
\coloneq
f\circ v - f^{j}\circ v.
$
Thus, $w$ is absolutely continuous. Applying Stampacchia's lemma yields
\begin{equation}\label{eq:wstampacchia}
\frac{d}{dt} w(t)=0
\quad
\text{a.e.} \quad t \in \{s \in [0,1] : w(s)=0 \}.
\end{equation}
By~\eqref{eq:equalityonpartitionfonction}, we have
$
U^{j}
\subset
\{s \in [0,1] : w(s)=0 \}.
$
So, from~\eqref{eq:wstampacchia} and~\eqref{eq:standardchainrule}, it follows that
\begin{equation}\label{eq:chainruleonpartition}
\frac{d}{dt} f(v(t))
=
\frac{d}{dt} f^{j}(v(t))
=
\langle \dot{v}(t), {\nabla} f^{j}(v(t)) \rangle
\quad
\text{a.e.} \quad t \in U^{j}.
\end{equation}

Moreover,
\begin{equation}\label{eq:indicatorequality}
\indicator_{U^{j}}(t)
=
\indicator_{P^{j}}(v(t))
\quad
\forall t \in [0,1].
\end{equation}

\noindent\emph{Third step.}
Consider
$
A(t)
\coloneq
\frac{d}{dt} f(v(t))
-
\langle \dot{v}(t), \widetilde{\nabla} f(v(t)) \rangle.
$
Then for almost all $t \in [0,1]$,
$$
\begin{aligned}
A(t)
&=
\frac{d}{dt} f(v(t))
-
\left\langle 
\dot{v}(t),
\sum_{j=1}^{{J}}
\indicator_{P^{j}}(v(t))
\nabla f^{j}(v(t))
\right\rangle
\\
&=
\left(
\sum_{j=1}^{{J}}
\indicator_{P^{j}}(v(t))
\right)
\frac{d}{dt} f(v(t))
-
\left\langle
\dot{v}(t),
\sum_{j=1}^{{J}}
\indicator_{P^{j}}(v(t))
\nabla f^{j}(v(t))
\right\rangle
\\
&=
\sum_{j=1}^{{J}}
\indicator_{P^{j}}(v(t))
\left(
\frac{d}{dt} f(v(t))
-
\left\langle \dot{v}(t),
\nabla f^{j}(v(t))
\right\rangle
\right).
\end{aligned}
$$
Using~\eqref{eq:indicatorequality}, we get, for almost all $t\in[0,1]$,
\begin{align*}
    A(t)
    =
\sum_{j=1}^{{J}}
\indicator_{U^{j}}(t)
\left(
\frac{d}{dt} f(v(t))
-
\left\langle \dot{v}(t), \nabla f^{j}(v(t)) \right\rangle
\right)
= 0,
\end{align*}
where the last equality comes from~\eqref{eq:chainruleonpartition}. 
\end{proof}

Next, we derive the following generalized Morse--Sard lemma from \cite[Theorem~5]{barbet2013morse}. It will be used in the algorithmic analysis of Section~\myref{sect:application}.

\begin{lemma}[Generalized Morse--Sard]\label{lem:morsesard}
Let $n \in \N \setminus \{0\}$. Let $f_i\colon\Rm\to\R$, $i\in\{1,\ldots,n\}$, be continuous and piecewise-$C^p$. For each $i$, let $\{P_i^j\}_{j=1}^{J_i}$ and $\{f_i^j\}_{j=1}^{J_i}$ define a representation
of $f_i$ as in Definition~\myref{def:piecewisec1}, and let $D_{f_i}$ be the set-valued map associated with this representation by \eqref{eq:defconsvfield}. Assume that $p\geq m$, and set
$$
f\coloneq\frac{1}{n}\sum_{i=1}^n f_i.
$$
Let $D_f$ be the set-valued map defined by \eqref{eq:defconsvfield} from
the common-refinement representation of $f$ induced by the representations of the functions $f_i$. The set of conservative critical values of $f$ has measure zero:
\begin{align*}
\lambda^1\!\left(
f\big(\{x\in\Rm:0\in\conv D_f(x)\}\big)
\right)&=0.
\end{align*}
Moreover, the set of Clarke critical values of $f$ has measure zero:
$$
\lambda^1\!\left(
f\big(\{x\in\Rm:0\in\partial^c f(x)\}\big)
\right)=0.
$$
Here, $\lambda^1$ denotes the Lebesgue measure on $\R$. In general,
$\lambda^m$ denotes the Lebesgue measure on $\Rm$.
\end{lemma}

\begin{proof}
Let
$ \displaystyle
\mathcal{J}\coloneq\prod_{i=1}^n\{1,\ldots,J_i\}.
$
Define
$\displaystyle
P^{\mathbf{j}}\coloneq\bigcap_{i=1}^nP_i^{j_i},
g^{\mathbf{j}}\coloneq\frac{1}{n}\sum_{i=1}^nf_i^{j_i}
$
for $\mathbf{j}=(j_1,\ldots,j_n)\in\mathcal{J}$.
After discarding the empty sets, the sets $P^{\mathbf{j}}$ form the common
refinement. Consequently,
$$
D_f(x)
=
\left\{
\nabla g^{\mathbf{j}}(x):
x\in \cl P^{\mathbf{j}},\ 
\mathbf{j}=(j_1,\ldots,j_n)\in\mathcal{J}
\right\}.
$$
The function $f$ is a continuous selection of the finite family of $C^p$
functions $\{g^{\mathbf{j}}\}_{\mathbf{j}\in\mathcal{J}}$. Define its active
index set and the corresponding active-gradient set by
$$
I(x)\coloneq
\{\mathbf{j}\in\mathcal{J}:g^{\mathbf{j}}(x)=f(x)\},
\quad
A_f(x)\coloneq
\conv\,\{\nabla g^{\mathbf{j}}(x):\mathbf{j}\in I(x)\}.
$$
If $x\in\cl \bigcap_{i=1}^n P_i^{j_i}$, then continuity and the equality
$f_i=f_i^{j_i}$ on $P_i^{j_i}$ imply
$f_i(x)=f_i^{j_i}(x)$ for every $i$. Hence
$g^{\mathbf{j}}(x)=f(x)$, and
$
 D_f(x)
\subset
 \{\nabla g^{\mathbf{j}}(x):\mathbf{j}\in I(x)\}.
$
The proof of \cite[Theorem 5]{barbet2013morse} shows that
$
\lambda^1\!\left(
f\big(\{x\in\Rm:0\in A_f(x)\}\big)
\right)=0.
$
The assertion concerning $D_f$ follows
from the preceding inclusion. Finally,
$\partial^c f(x) \subset A_f(x)$ by
\cite[Proposition 4]{barbet2013morse}, which proves the assertion concerning
the Clarke critical values.
\end{proof}

We conclude this section with a result that depends on how the parts fit together. It is used only in Theorem~\myref{theo:clarkcrit} of Section~\myref{sect:application}.
\begin{lemma}\label{lem:piecewiseCplocallyC2ae}
Let $p\in \N \cup \{\infty\}$. Let $f\colon \Rm \to \R$ be continuous and piecewise-$C^{p}$, and let $\{P^{j}\}_{j=1}^{{J}}$ and $\{f^{j}\}_{j=1}^{{J}}$ denote a Borel partition and a collection of $C^{p}$ functions defining a representation of $f$ as in Definition~\myref{def:piecewisec1}. Define the interface by
$$
\mathcal{I}
\coloneq
\bigcup_{\substack{j,\ell\in\{1,\ldots,J\}\\j\neq \ell}}
\left(\cl{P^j}\cap\cl{P^{\ell}}\right).
$$
If $\mathcal{I}$ has measure zero, then $f$ is locally $C^p$ almost everywhere.
\end{lemma}

\begin{proof}
Let
$
x\in\Rm\setminus\mathcal{I}.
$

Since $\{P^{j}\}_{j=1}^{J}$ is a partition of $\Rm$, there exists a unique
$i\in\{1,\ldots,J\}$ such that
$
x\in P^i.
$
We claim that
$
x\in\interior P^i.
$\footnote{The notation $\interior A$ denotes the interior of the set $A$.}

Suppose, for contradiction, that
$
x\notin\interior P^i.
$
Since $x\in P^i$, every neighborhood of $x$ intersects
$\Rm\setminus P^i$. Hence, for all $k\geq1$, there exists
$
x_k\in B(x,1/k)\setminus P^i,
$
where $B(x,1/k) \coloneq \{y \in \Rm: \|x - y\| < 1/k\}$. Because $\{P^j\}_{j=1}^{J}$ is a partition, there exists
$\ell_k\neq i$ such that
$
x_k\in P^{\ell_k}.
$
Since there are only finitely many parts in the partition, one may extract a subsequence,
still denoted by $(x_k)_{k\geq 1}$, and an index $\ell\neq i$ such that
$
x_k\in P^{\ell}
$
for all $k \geq 1$.
Passing to the limit gives
$
x\in\cl{P^{\ell}}.
$
Since
$
x\in P^i\subset\cl{P^i},
$
it follows that
$
x\in\cl{P^i}\cap\cl{P^{\ell}}
\subset\mathcal{I},
$
which contradicts
$
x\notin\mathcal{I}.
$
Therefore,
$
x\in\interior P^i.
$

Hence, there exists $r>0$ such that
$
B(x,r) \subset P^i.
$
Consequently,
$
f=f^i\text{ on }B(x,r).
$
Since $f^i$ is $C^p$,
$\displaystyle
f|_{B(x,r)}
$
is $C^p$.
Thus, $f$ is locally $C^p$ on
$
\Rm\setminus\mathcal{I}.
$
The conclusion follows.
\end{proof}

\section{Piecewise smooth functions are path differentiable}\label{sect:main}
In this section, we present our main result. We start with a lemma that shows that $D_f$ in \eqref{eq:defconsvfield} has the required basic properties.

\begin{lemma}\label{lem:conservativeproperties} Let $f\colon \Rm \to \R$ be continuous and piecewise-$C^{p}$, and let $\{P^{j}\}_{j=1}^{{J}}$ and $\{f^{j}\}_{j=1}^{{J}}$ denote a Borel partition and a collection of $C^{p}$ functions defining a representation of $f$ as in Definition~\myref{def:piecewisec1}.

The set-valued map $D_f$, as defined in \eqref{eq:defconsvfield}, has a closed graph, nonempty and compact values, and is locally bounded.
\end{lemma}

\begin{proof} Because $f$ is piecewise-$C^{p}$, every $x \in \Rm$ lies in some part $P^j \subset \cl P^j$. So $\nabla f^j(x) \in D_f(x)$, and $D_f$ has nonempty values. Its values are finite sets, hence compact.

We then prove that $\graph D_f$ is closed. Let
$$
x_{\ell}\to x,\qquad u_{\ell}\to u,\qquad u_{\ell}\in D_f(x_{\ell}).
$$
For every $\ell$, there exists $j_{\ell}\in\{1,\ldots,{J}\}$ such that
$$
x_{\ell}\in \cl P^{j_{\ell}},
\qquad
u_{\ell}=\nabla f^{j_{\ell}}(x_{\ell}).
$$
Since the set of indices is finite, up to extracting a subsequence, we may assume that
$j_{\ell}=j$ for some fixed $j$. Hence
$$
x_{\ell}\in\cl P^{j}
\qquad\text{and}\qquad
u_{\ell}=\nabla f^{j}(x_{\ell}).
$$
Since $\cl P^{j}$ is closed and $x_{\ell}\to x$, we get 
$
x\in \cl P^{j}.
$
Moreover, by the continuity of $\nabla f^{j}$,
$$
u=\lim_{\ell\to\infty}u_{\ell}
=
\lim_{\ell\to\infty}\nabla f^{j}(x_{\ell})
=
\nabla f^{j}(x).
$$
Therefore, $u\in D_f(x)$, and so $\graph D_f$ is closed.

We now prove local boundedness. Fix $x_0\in \Rm$ and let $r>0$. Since each
$\nabla f^{j}$ is continuous, it is bounded on the compact set $B_c(x_0,r)\coloneq \{x\in \Rm: \|x_0 - x\|\leq r\}$.
Thus
$$
L(x_0, r) \coloneq
\max_{1\leq j \leq {J}}
\sup_{x\in B_c(x_0,r)}\|\nabla f^{j}(x)\|
< \infty.
$$
Let $x\in B_c(x_0,r)$ and $u \in D_f(x)$. Then for some $j \in\{1, \ldots, J\}$,
$
u=\nabla f^{j}(x),
$
and therefore
$
\|u\|\leq L(x_0, r).
$
Hence $D_f$ is locally bounded.
\end{proof}

\begin{theorem}[Main result]\label{theo:main} Let $f\colon \Rm \to \R$ be continuous and piecewise-$C^{p}$, and let $\{P^{j}\}_{j=1}^{{J}}$ and $\{f^{j}\}_{j=1}^{{J}}$ denote a Borel partition and a collection of $C^{p}$ functions defining a representation of $f$ as in Definition~\myref{def:piecewisec1}. Let $D_f$ be as defined in \eqref{eq:defconsvfield}. 

Then $D_f$ is a conservative field for $f$; that is, for any Lipschitz continuous curve $v\colon[0,1]\to\Rm$, for almost all $t \in [0,1]$,
    \begin{equation*}
    \frac{d}{dt} f(v(t)) = \langle \dot{v}(t), u \rangle \quad \forall u \in D_f(v(t)).
    \end{equation*}
In particular, $f$ is path differentiable, and $\conv D_f$ and $\partial^{c} f$ are conservative fields for $f$.
\end{theorem}
\begin{proof}

Let $v\colon [0,1] \to \Rm$ be Lipschitz continuous.

For every pair of indices $i,j\in\{1,\ldots,{J}\}$, define
$$
E_{ij}
\coloneq
\left\{
s\in[0,1]:
f^i(v(s))=f^j(v(s))
\right\}.
$$

By Stampacchia's result in Lemma~\myref{lem:stampacchia}, since
$f^i\circ v$ and $f^j\circ v$ are absolutely continuous and coincide on
$E_{ij}$, one has
$$
 \langle \dot v(s), \nabla f^i(v(s)) \rangle
=
 \langle \dot v(s), \nabla f^j(v(s)) \rangle
\quad
\text{ a.e.} \quad s\in E_{ij}.
$$

For all $i,j \in \{1,\ldots,{J}\}$, let $N_{ij} \subset E_{ij}$ with measure zero such that
\begin{equation*}
    E_{ij} \setminus N_{ij} \subset \{s \in E_{ij}: \langle \dot v(s), \nabla f^i(v(s)) \rangle
=
 \langle \dot v(s), \nabla f^j(v(s)) \rangle\}.
\end{equation*}
Define $\displaystyle \Omega_1 \coloneq [0,1] \setminus \bigcup_{1\leq i,j \leq {J}} N_{ij}$. $\Omega_1$ has full measure in $[0,1]$ and for every $s\in \Omega_1$ and every
$i,j\in\{1,\ldots,{J}\}$,

\begin{equation}\label{eq:equalityset}
s\in E_{ij}
\implies
\langle \dot v(s), \nabla f^i(v(s)) \rangle
=
\langle \dot v(s), \nabla f^j(v(s)) \rangle.
\end{equation}
Let $\Omega_2 \subset [0,1]$ be a full-measure set where the Murat--Trombetti chain rule in Theorem~\myref{theo:murat} holds for $f \circ v$. Set $\Omega = \Omega_1 \cap \Omega_2$. $\Omega$ has full measure in $[0,1]$.

Let $t \in \Omega$ and let $ u \in D_f(v(t)) $. Then there exists $ i $ such that
$$
u = \nabla f^i(v(t)) \quad \text{with} \quad v(t) \in \cl P^i.
$$

Let $ j $ be the unique index such that $ v(t) \in P^{j} $. We need to show
$$
\langle \dot{v}(t), \nabla f^i(v(t)) \rangle = \langle \dot{v}(t), \nabla f^{j}(v(t)) \rangle.
$$

Since $ v(t) \in \cl P^i $, take $ (y_{\ell})_{\ell \in \N} \subset P^i $ with $ y_{\ell} \to v(t) $. Because $ f = f^i $ on $ P^i $ and $ f $ is continuous,
$$
f^i(v(t)) = \lim_{\ell} f^i(y_{\ell}) = \lim_{\ell} f(y_{\ell}) = f(v(t)).
$$

Also, since $ v(t) \in P^{j} $,
$$
f^{j}(v(t)) = f(v(t)) = f^i(v(t)).
$$
So $t \in E_{ij}$ and from \eqref{eq:equalityset}, we have
$$
\langle \dot{v}(t), \nabla f^i(v(t)) \rangle = \langle \dot{v}(t), \nabla f^{j}(v(t)) \rangle.
$$

Thus, 
$$
\langle \dot{v}(t), u \rangle = \langle \dot{v}(t), \nabla f^{i}(v(t)) \rangle = \langle \dot{v}(t), \nabla f^{j}(v(t)) \rangle = \langle \dot{v}(t), \widetilde{\nabla} f(v(t)) \rangle = \frac{d}{dt } f(v(t)),
$$

where the last equality follows from Theorem~\myref{theo:murat}. Thus,
\begin{equation}\label{eq:chainruleproof}
\frac{d}{dt} f(v(t)) = \langle \dot{v}(t), u \rangle \quad \forall u \in D_f(v(t)).   
\end{equation}
Moreover, thanks to Lemma~\myref{lem:conservativeproperties}, we know that $D_f$ has closed graph, nonempty and compact values, and is locally bounded. Therefore, together with~\eqref{eq:chainruleproof}, $D_f$ is conservative for $f$, and so are $\conv D_f$ and $\partial^{c} f$; see \cite[Remark~3~(e) and Corollary~1]{bolte2021conservative}.
\end{proof}
\begin{remark}
    In Remark~\myref{rmk:selection}, we saw that the piecewise gradient of $f$ is a selection of $D_f$, which by Theorem~\myref{theo:main} is a conservative field for $f$. Furthermore, Theorem~\myref{theo:main} shows that, for all $p \in \left(\N \setminus \{0\}\right)\cup \{\infty\}$, the class of continuous piecewise-$C^p$ functions is a subset of the class of path-differentiable functions.
\end{remark}

\begin{definition}[Semismooth generalized gradient~{\cite{ermol1998stochastic}}]
    A locally bounded, graph-closed set-valued map $D \colon \Rm \rightrightarrows \Rm$ with nonempty, compact, and convex values is a Norkin or semismooth generalized gradient of a locally Lipschitz function $\varphi\colon \Rm \to \R$ if
    $$
    \sup_{\xi\in D(x+h)}
    \left|
    \varphi(x+h)-\varphi(x)-\langle \xi,h\rangle
    \right|
    =o(\|h\|)
    \quad \text{as} \quad h\to0.
    $$
\end{definition}

Next, we show that the convex hull $\conv D_f$ of $D_f$, which is also conservative for $f$, is a Norkin or semismooth generalized gradient of $f$.

\begin{proposition}[$\conv D_f$ is a semismooth generalized gradient]
\label{prop:piecewise-cp-semismooth2}
Let $f\colon \Rm \to \R$ be continuous and piecewise-$C^{p}$, and let $\{P^{j}\}_{j=1}^{{J}}$ and $\{f^{j}\}_{j=1}^{{J}}$ denote a Borel partition and a collection of $C^{p}$ functions defining a representation of $f$ as in Definition~\myref{def:piecewisec1}.
Let $D_f$ be as defined in \eqref{eq:defconsvfield}. 

Then the convex-valued conservative field
$
\conv D_f
$
for $f$
is a semismooth generalized gradient of $f$.
\end{proposition}
\begin{proof}
Define
$
I(x)
\coloneq
\{j\in\{1,\ldots,J\}:x\in\cl P^j\}.
$
Then
$
D_f(x)
=
\{\nabla f^j(x):j \in I(x)\}.
$
Continuity of $f$ implies that
\begin{equation}
\label{eq:activeselectionagreement}
f^j(x)=f(x)
\quad
\text{whenever }j \in I(x).
\end{equation}
Finiteness of the partition implies that, for every
$x\in\Rm$, there exists a neighborhood $U$ of $x$
such that
\begin{equation}
\label{eq:activesetinclusion}
I(y) \subset I(x)
\quad
\text{for every }y \in U.
\end{equation}
Fix
$x\in\Rm$, let $y=x+h\in U$, and take
$\xi\in \conv D_f(y)$. Write
$$
\xi
=
\sum_{j\in I(y)}
\lambda_j\nabla f^j(y),
\qquad
\lambda_j\geq0,
\qquad
\sum_{j\in I(y)}\lambda_j=1.
$$
For every $j\in I(y)$, Equations
\eqref{eq:activeselectionagreement} and
\eqref{eq:activesetinclusion} give
$$
f^j(y)=f(y)
\qquad\text{and}\qquad
f^j(x)=f(x).
$$
Consequently,
$$
\begin{aligned}
f(y)-f(x)-\langle\xi,h\rangle
&=
\sum_{j\in I(y)}
\lambda_j
\left(
f^j(y)-f^j(x)
-\langle\nabla f^j(y),h\rangle
\right).
\end{aligned}
$$
For each $j$, the fundamental theorem of calculus yields
\begin{align*}
\left|f^j(y)-f^j(x)-\langle\nabla f^j(y),h\rangle \right|
&\leq
\int_0^1
\left|\left\langle
\nabla f^j(x+th)-\nabla f^j(y),h
\right\rangle\right|\,dt  \\
&\leq \|h\|\max_{j \in I(x)} \sup_{t \in [0,1]} \left\|\nabla f^j(x+th)-\nabla f^j(y)\right\|.
\end{align*}
Because the functions $f^{j}$ are finite and their
gradients are continuous, the right-hand side is $o(\|h\|)$ uniformly in $j \in I(x)$. 
So $\conv D_f$ is a semismooth
generalized gradient of $f$.
\end{proof}

Despite the previous result, the example below shows that, for piecewise smooth functions, convex-valued conservative fields are not, in general, semismooth generalized gradients.
\begin{example}
\label{ex:conservative-not-norkin}
Let
$
f\colon \R\to\R,
\;
f(x)= 1 \text{ for all } x \in \R,
$
and define the closed countable set
$
S\coloneq\{0\}\cup\left\{\frac{1}{n}:n\in\N\setminus\{0\}\right\}.
$
$f$ is trivially $C^{\infty}$ everywhere, hence piecewise smooth.
Consider the set-valued map
$$
D(x)\coloneq
\begin{cases}
[-3,2], & x\in S,\\
\{0\},  & x\notin S.
\end{cases}
$$
The mapping $D$ is locally bounded, graph-closed, and has
nonempty, compact, and convex values. We claim that $D$ is a conservative field for $f$. Indeed, let
$v:[0,1]\to\R$ be Lipschitz. Since $S$ is
countable,
$$
\dot{v}(t)=0
\quad
\text{a.e.} \quad t \in \{r \in [0,1]: v(r) \in S\}.
$$
This follows from Lemma~\myref{lem:stampacchia} applied to $t \mapsto v(t) - s$ for all $s\in S$ to get
$$
\dot{v}(t) = 0 \quad \text{a.e.}\quad t \in \{r \in [0,1]: v(r) = s\},
$$
and then taking the countable union over $S$. Therefore,
for almost all $t\in[0,1]$ and all $u\in D(v(t))$,
$$
\langle \dot{v}(t), u \rangle = 0
=\frac{d}{dt} f(v(t)).
$$
Thus $D$ is conservative for $f$. Nevertheless, $D$ is not a generalized gradient in the sense of Norkin. Indeed, take
$
x=0,
x_n=\frac{1}{n},
u_n=-2\in D(x_n).
$
Then
$$
\frac{
\left|
f(x_n)-f(0)-u_n(x_n-0)
\right|
}{
|x_n|
}
=2.
$$
So the Norkin semismoothness condition fails.
\end{example}

\section{Algorithmic application}\label{sect:application}
In this section, we analyze the dynamics of a stochastic first-order method driven by piecewise gradients, referred to as the stochastic piecewise-gradient method, for solving the following optimization problem:
\begin{equation*}
    \minimize_{x \in \Rm} \big \{f(x) \coloneq \frac{1}{n}\sum_{i=1}^{n} f_{i}(x)\big\},
\end{equation*}
where, for all $i \in \{1, \ldots, n\}$, $f_{i}\colon \Rm \to \R$ is continuous and piecewise-$C^{p}$, and $\displaystyle \{P^{j}_{i}\}_{j=1}^{{J_i}}$ and $\displaystyle \{f^{j}_{i}\}_{j=1}^{{J_i}}$ denote a Borel partition and a collection of $C^p$ functions defining a representation of $f_i$ as in Definition~\myref{def:piecewisec1}. By Lemma~\myref{lem:vectorspace}, $f$ is itself continuous and piecewise-$C^p$. We assume that $f$ is bounded below.

Consider the stochastic piecewise-gradient method:
\begin{equation}
    x_{k + 1} = x_{k} - c\alpha_{k} \widetilde{\nabla}f_{i_{k}}(x_{k}),
    \label{eq:gd}
\end{equation}
where $\alpha_k > 0$ for all $k \in \N$, $c>0$, $x_0 \in \Rm$ and $(i_k)_{k\in\N}$ is a sequence of random variables taking values in $\{1, \dots, n\}$. 

From Lemma~\myref{lem:vectorspace}, it is clear that $\displaystyle \widetilde{\nabla} f(x) \coloneq \frac{1}{n}\sum_{i=1}^{n}\widetilde{\nabla} f_i(x)$ is the piecewise gradient of $f$ corresponding to the common-refinement representation $
\left(\{P^j\}_{j=1}^{J},\{f^j\}_{j=1}^{J}\right)
$ of $f$ induced by the representations of the functions $f_i$. Let $D_f$ be the set-valued map constructed according to~\eqref{eq:defconsvfield} from this same representation of $f$.
Thus, for all $x \in \Rm$, 
\begin{equation}\label{eq:associatedinclusion}
 \widetilde{\nabla} f(x) = \frac{1}{n}\sum_{i=1}^{n}\widetilde{\nabla} f_i(x) \in D_f(x),
\end{equation}
and, by Theorem~\myref{theo:main}, $D_f$ is a conservative field for $f$.

\paragraph{Setting of the probability spaces and random variables.}
Equip $\Rm$ and $\R$ with their respective Borel $\sigma$-algebras
$\mathcal{B}(\Rm)$ and $\mathcal{B}(\R)$, and $\{1,\ldots,n\}$ with the power set $\mathcal{P}(\{1,\ldots,n\})$. Let
$
(\Omega_0,\mathcal{F}_0,\mathbb{P}^0)
$
be an abstract probability space. Let
$
\iota \colon \Omega_0 \to \{1,\ldots,n\}
$
be a random variable satisfying
$
\mathbb{P}^0(\iota=\ell)=\frac{1}{n},
\ell\in\{1,\ldots,n\}.
$
Consider the product probability space
$$
\Omega
=
\Omega_0^{\N},
\quad
\mathcal{F}
=
\mathcal{F}_0^{\otimes \N},
\quad
\mathbb{P}
=
(\mathbb{P}^0)^{\otimes \N}.
$$
For every
$
\omega=(\omega_0,\omega_1,\omega_2,\ldots)\in\Omega,
$
define 
$
i_k\colon \Omega \to \{1, \ldots, n\}
$
by
$
i_k(\omega)\coloneq\iota(\omega_{k})
$
for all $k\in\N$.
\begin{remark}
    By construction, $(i_k)_{k\in\N}$ is a sequence of i.i.d. random variables, each uniformly distributed on $\{1,\ldots,n\}$.
\end{remark}
Recall that, for $S \subset \Rm$ and $x \in \Rm$, $\displaystyle \dist(x, S) \coloneq \inf\{\|x - y\|: y \in S\}$. By convention, $\dist(x, \varnothing) = \infty$.

In the next theorems, we prove subsequence convergence of \eqref{eq:gd} to both conservative and Clarke critical points, as well as convergence of the function values.

\begin{theorem}[Conservative critical limit points]\label{theo:convconservative}
    Suppose that $p \geq m \geq 1$. Let $x_0 \in \Rm$. Suppose that $\alpha_k = o(1/\log(k))$ as $k \to \infty$ and $\sum_{k = 0}^{\infty} \alpha_k = \infty$. Assume that the iterates $(x_k)_{k\in\N}$ generated by \eqref{eq:gd} with $c=1$ are bounded $\mathbb{P}$-almost surely. Then $\mathbb{P}$-almost surely,
    \begin{equation*}
        \lim_{k\to \infty} \dist\left(x_k, \{x \in \Rm: 0 \in \conv D_f(x)\} \right)  = 0,
    \end{equation*}
and the sequence of function values $(f(x_k))_{k \in \N}$ converges. Moreover, the set of limit points of $(x_k)_{k\in\N}$ is connected.
\end{theorem}
\begin{proof} 
    We can rewrite the iteration~\eqref{eq:gd} as
    \begin{equation*}
        x_{k+1} = x_{k} -\alpha_k \widetilde{\nabla} f(x_k) - \alpha_k \eta_k, \quad\text{where } \eta_k = \widetilde{\nabla}f_{i_k}(x_k) - \widetilde{\nabla} f(x_k).
    \end{equation*}
    For all $k \in \N$, $\mathbb{P}$-almost surely, we have $\mathbb{E}\left[\widetilde{\nabla}f_{i_k}(x_k) \mid \mathfrak{F}_{k}\right] = \widetilde{\nabla} f(x_k)$ and $\mathbb{E}[\eta_{k} \mid \mathfrak{F}_{k}] = 0$, where $\mathfrak{F}_k$ is the $\sigma$-algebra generated by $\{x_0, i_0, \ldots, i_{k-1}\}$. 
    
    From \eqref{eq:associatedinclusion}, we recall that $\widetilde{\nabla} f(x) \in \conv D_f(x)$ for all $x \in \Rm$, and $\conv D_f$ is a conservative field for $f$.
    
    Since Theorem~\myref{theo:main} shows that the piecewise gradients $\widetilde{\nabla} f_i$ are selections of the conservative fields $D_{f_i}$, we follow the argument of \cite{benaim2005stochastic}, where the continuous flow
    \begin{equation}\label{eq:flow}
        \dot{x}(t) \in - \conv D_f(x(t)),
    \end{equation}
    is used to analyze the dynamics of $(x_k)_{k\in\N}$. This approach, often called the ODE method, is standard in stochastic approximation;  see \cite{bolte2023subgradient, le2024nonsmooth, xiao2024adam, davis2020stochastic} and the references therein.
    
    Because $\conv D_f$ has convex values, is conservative for $f$, and $f$ is bounded below, \cite[Remark 3.11]{bolte2023subgradient} ensures that \eqref{eq:flow} has a solution on $[0,\infty)$. That is, for all $x_0 \in \Rm$, there exists an absolutely continuous curve $x\colon [0,\infty) \to \Rm$ such that $x(0) = x_0$ and $\dot{x}(t) \in - \conv D_f(x(t))$ for almost all $t\in[0,\infty)$.
    
    The function $\psi$ defined by 
    $\displaystyle
    \psi(x)\coloneq\max_{1\leq i\leq n}\max_{y \in D_{f_i}(x)}\left\| y\right\|
    $
    is locally bounded. Therefore, all conditions of Assumption~7 in \cite{bolte2023subgradient} are satisfied in our setting, and \cite[Lemma~3.14~(3)]{bolte2023subgradient} gives
    \begin{equation}\label{eq:condbenaim}
        \forall T > 0, \quad \lim_{k\to\infty} \sup_{N \geq k} \left\{ \left\| \sum_{\ell=k}^{N} \alpha_{\ell} \eta_{\ell} \right\| : \sum_{\ell=k}^{N} \alpha_{\ell} \leq T \right\} = 0  \quad \mathbb{P}\text{-a.s.}
    \end{equation}
    Let $y\colon [0,\infty) \to \Rm$ be the continuous-time \emph{affine interpolated process} associated with $(x_k)_{k\in\N}$ as defined in \cite[Definition~$\mathrm{IV}$]{benaim2005stochastic}. Items (i) and (ii) of Proposition~1.3 in \cite{benaim2005stochastic} are given by \eqref{eq:condbenaim} and the almost-sure boundedness assumption of the iterates, respectively. Thus, according to \cite[Proposition~1.3]{benaim2005stochastic}, $y$ is a perturbed solution to \eqref{eq:flow} almost surely. 
    By \cite[Theorems~4.2~and~4.3]{benaim2005stochastic}, the limit set of $y$,
    $$
    L(y)\coloneq\bigcap_{t\geq0}\cl\left\{y(s): s \geq t\right\},
    $$
    is, $\mathbb{P}$-a.s., internally chain transitive; see \cite[Definition~$\mathrm{VI}$]{benaim2005stochastic}. In addition, $f$ is a Lyapunov function for the critical set $\Lambda \coloneq \{x \in \Rm: 0 \in \conv D_f(x)\}$ and for the flow \eqref{eq:flow} (see \cite[Lemma~3.13]{bolte2023subgradient}), and by Lemma~\myref{lem:morsesard}, $f(\Lambda)$ has empty interior. Thus, from \cite[Proposition~3.27]{benaim2005stochastic}, we get that $f$ is constant on $L(y)$  and $L(y) \subset \Lambda$, $\mathbb{P}$-a.s. Since $\psi$ is locally bounded, $\|x_{k+1} - x_k\| \to 0$ $\mathbb{P}$-a.s., which implies that the set of limit points of $(x_k)_{k\in\N}$ is exactly $L(y)$ and is a connected subset of $\Lambda$. This completes the proof.
\end{proof}
\begin{theorem}[Clarke critical limit points]\label{theo:clarkcrit}
Suppose that $p \geq \max\{m, 2\}$. Assume that $(\alpha_k)_{k\in\N}$ satisfies 
$
\sum_{k=0}^{\infty}\alpha_k=\infty
$
 and 
$
\alpha_k = o(1/\log(k))
$ as $k \to \infty$.
Suppose also that, for all $i \in \{1, \ldots, n\}$, the interface
\begin{equation}\label{eq:interface}
\mathcal{I}_i
\coloneq
\bigcup_{\substack{j,\ell\in\{1,\ldots,J_i\}\\j\neq \ell}}
\left(\cl{P^j_i}\cap\cl{P^{\ell}_i}\right)\; \text{ is Lebesgue negligible.}
\end{equation}
Then there exists $\Gamma\subset(0,\infty)$ of full Lebesgue measure such that, for Lebesgue-almost every $c > 0$,
$
    c\alpha_k \in\Gamma
    \text{ for all } k\in\N.
$
For each such $c > 0$, there exists $G \subset \Rm$ of full Lebesgue measure with the following property: if $x_0 \in G$ and the iterates $(x_k)_{k\in\N}$ generated by \eqref{eq:gd} are bounded $\mathbb{P}$-almost surely, then $\mathbb{P}$-almost surely,
$$
    \lim_{k\to\infty}
    \dist\!\left(
        x_k,
        \{x\in\Rm:0\in\partial^c f(x)\}
    \right)
    =0 ,
$$
and the sequence of function values $(f(x_k))_{k\in\N}$ converges.  Moreover, the set of limit points of $(x_k)_{k\in\N}$ is connected.
\end{theorem}
\begin{remark}\label{rmk:relu-interface}
The interface assumption in \eqref{eq:interface} is satisfied whenever each relevant boundary has lower dimension or, more generally, measure zero.
Typical examples include fully connected neural networks with piecewise-affine activations such as ReLU; see \cite[Lemma~1 and Theorem~1]{pmlr-v80-laurent18b} and \cite[Corollary~33 and Definition~17]{bona2023parameter}. Another instance where it is satisfied is when the partitions are stratifications with a frontier condition: 
$$
\forall \ell \neq j, \quad P^j_i\cap \cl P^{\ell}_i \neq \varnothing \implies P^j_i \subset (\cl P^{\ell}_i) \setminus P^{\ell}_i \text{ and } \dim P^j_i < \dim P^{\ell}_i;
$$
see an example in Remark~\myref{rmk:partitionexample}. 
 
The assumption \eqref{eq:interface} is needed to ensure, thanks to Lemma~\myref{lem:piecewiseCplocallyC2ae}, that $f_i$ is locally $C^2$ almost everywhere, which is a key ingredient in the proof of Theorem~\myref{theo:clarkcrit}. A merely Lipschitz continuous and piecewise-$C^\infty$ function need not be locally $C^1$ on a set of positive measure; see Proposition~\myref{claim:counter}, where we construct such a function.
\end{remark}
\begin{proof}[Proof of Theorem~\myref{theo:clarkcrit}] 
This proof builds on key ideas from \cite{bianchi2022convergence}. It
uses less probabilistic machinery---for example, it does not use Markov kernels---and is adapted to accommodate decreasing stepsizes and a full-measure set of deterministic initializations\footnote{It also allows random initializations according to distributions that are absolutely continuous with respect to the Lebesgue measure.}. We divide the proof into four steps.

Lemma~\myref{lem:piecewiseCplocallyC2ae} implies that each $f_i$ is locally $C^2$ Lebesgue-almost everywhere. Thus, for every $i\in\{1,\dots, n\}$, there exists an open set
$O_i\subset\Rm$ such that
$
\lambda^{m}(\Rm\setminus O_i)=0
$
and
$
f_i|_{O_i}
$
is $C^2$ on $O_i$.
For $\alpha>0$, define
$
\psi_{i,\alpha}(x)
=
x-\alpha\nabla f_i(x),
 \text{ if } x\in O_i
$
and
$
\psi_{i,\alpha}(x)
=
x,
 \text{ otherwise}.
$
Note that $\psi_{i,\alpha}$ is measurable.

\noindent \emph{{First step.}} We start by constructing $\Gamma$, a set of full Lebesgue measure in $(0, \infty)$, such that, for any $\alpha \in \Gamma$ and any $i\in\{1, \ldots, n\}$,  the derivative of $\psi_{i,\alpha}$ is invertible Lebesgue-almost everywhere on $O_i$.

Fix $i\in\{1,\ldots,n\}$. For $(x,\alpha)\in O_i\times(0,\infty)$,
define
$
h_i(x,\alpha)
\coloneq
\det\left(\operatorname{Id}_m-\alpha\nabla^2 f_i(x)\right).
$
Since $f_i$ is $C^2$ on $O_i$, the map
$
(x,\alpha)\mapsto h_i(x,\alpha)
$
is continuous, and therefore measurable.
Consider the measurable set
$
S_i
\coloneq
\left\{
(x,\alpha)\in O_i\times(0,\infty):
h_i(x,\alpha)=0
\right\}.
$
For fixed $x\in O_i$, let
$
\mu_1(x),\dots,\mu_m(x)
$
denote the eigenvalues of the symmetric matrix $\nabla^2 f_i(x)$.
Then
$
h_i(x,\alpha)
=
\prod_{j=1}^m\left(1-\alpha\mu_j(x)\right).
$
Hence
$
h_i(x,\alpha)=0
$
only if
$
\alpha=\frac{1}{\mu_j(x)}
$
for some $j$ such that $\mu_j(x)>0$. Consequently, the section
$
S_i^x
\coloneq
\left\{
\alpha>0:(x,\alpha)\in S_i
\right\}
$
contains at most $m$ points. In particular,
$
\lambda^1\left(S_i^x\right)=0
\;
\text{ for all }x\in O_i.
$
Applying the Fubini--Tonelli theorem to the indicator function
$\mathbf{1}_{S_i}$ yields
$$
(\lambda^{m}\otimes\lambda^1)(S_i)
=
\int_{O_i}
\lambda^1\left(S_i^x\right)\,
d\lambda^{m}(x)
=0.
$$
Using the Fubini--Tonelli theorem in the reverse order, we obtain
$
0
=
\int_0^\infty
\lambda^{m}(S_i^\alpha)\,
d\alpha,
$
where
$$
S_i^\alpha
\coloneq
\left\{
x\in O_i:(x,\alpha)\in S_i
\right\}
=
\left\{
x\in O_i:
\det\left(\operatorname{Id}_m-\alpha\nabla^2f_i(x)\right)=0
\right\}.
$$
Since the integrand is nonnegative, it follows that
$
\lambda^{m}(S_i^\alpha)=0
\text{ for almost all }\alpha>0.
$
Therefore
$
{\Gamma}_i
\coloneq
\left\{
\alpha>0:
\lambda^{m}(S_i^\alpha)=0
\right\}
$
has full Lebesgue measure in $(0, \infty)$. We then set $\Gamma \coloneq \bigcap_{i=1}^{n} \Gamma_i$. $\Gamma$ has full Lebesgue measure in $(0, \infty)$.

\noindent \emph{{Second step.}} We show that, for any Lebesgue-measurable $A \subset \Rm$,
\begin{equation*}
    \lambda^{m}(A)=0 \implies \lambda^{m}\left(\psi_{i,\alpha}^{-1}(A)\right) = 0 \quad \forall i \in \{1, \ldots, n\} \quad \forall \alpha \in \Gamma.
\end{equation*}
Fix $\alpha\in {\Gamma}$. Let $i \in \{1, \ldots, n\}$ and define
$$
U_{i,\alpha}
\coloneq
O_i\setminus S_i^\alpha
=
\left\{
x\in O_i:
\det\left(\operatorname{Id}_m-\alpha\nabla^2f_i(x)\right)\neq0
\right\}.
$$
The set $U_{i,\alpha}$ is open because the map
$
x\mapsto
\det\left(\operatorname{Id}_m-\alpha\nabla^2f_i(x)\right)
$
is continuous on the open set $O_i$.
Moreover,
$$
\Rm\setminus U_{i,\alpha}
\subset
(\Rm\setminus O_i)\cup S_i^\alpha,
$$
and both sets on the right-hand side are Lebesgue-null. Therefore
$
\lambda^{m}(\Rm\setminus U_{i,\alpha})=0.
$
For $x\in U_{i,\alpha}$,
$
\nabla \psi_{i,\alpha}(x)
=
\operatorname{Id}_m-\alpha\nabla^2f_i(x),
$
and this matrix is invertible. By the inverse function theorem, for
every $x\in U_{i,\alpha}$ there exists an open neighborhood
$V_x\subset U_{i,\alpha}$ such that
$
\psi_{i,\alpha}|_{V_x}
\colon
V_x\to \psi_{i,\alpha}(V_x)
$
is a $C^1$-diffeomorphism. Since $\Rm$ is second countable, there exists a countable family of such neighborhoods $(V_{\ell})_{\ell \in \N}$ satisfying
$
U_{i,\alpha}
=
\bigcup_{\ell=0}^{\infty} V_{\ell}.
$

Let $A\subset\Rm$ be a Lebesgue-measurable set such that
$
\lambda^{m}(A)=0.
$
For each $\ell \in \N$, the inverse map
$$
\left(\psi_{i,\alpha}|_{V_{\ell}}\right)^{-1}
\colon
\psi_{i,\alpha}(V_{\ell})\to V_{\ell}
$$
is $C^1$, hence locally Lipschitz. Since locally Lipschitz maps send
Lebesgue-null sets to Lebesgue-null sets,
$
\lambda^{m}
\left(
\psi_{i,\alpha}^{-1}(A)\cap V_{\ell}
\right)
=0
$
for any $\ell \in \N$.
Thus,
$
\lambda^{m}
\left(
\psi_{i,\alpha}^{-1}(A)\cap U_{i,\alpha}
\right)
=0.
$
Since $\Rm\setminus U_{i,\alpha}$ is null,
$
\lambda^{m}
\left(
\psi_{i,\alpha}^{-1}(A)
\right)
=0.
$ Hence
\begin{equation}\label{eq:abscontdeter}
    \lambda^{m}(A)=0 \implies \lambda^{m}\left(\psi_{i,\alpha}^{-1}(A)\right) = 0 \quad \forall i \in \{1, \ldots, n\} \quad \forall \alpha \in \Gamma.
\end{equation}
Before moving to the third step, let
$
\Gamma^{c} \coloneq (0, \infty) \setminus \Gamma,
$
which is Lebesgue negligible. For each $ k $, the bad scalings are
$
\Gamma^{c}_k \coloneq \{ c > 0 : c \alpha_k \in \Gamma^{c} \} = \frac{1}{\alpha_k} \Gamma^{c}.
$
Since multiplication by $ 1/\alpha_k $ preserves negligible sets,
$
\lambda^1(\Gamma^{c}_k) = 0.
$
Therefore,
$
\Gamma^{c}_* \coloneq \bigcup_{k=0}^{\infty} \Gamma^{c}_k
$
is Lebesgue negligible.
For every
$
c \in (0, \infty) \setminus \Gamma^{c}_*,
$
$
 c \alpha_k \in \Gamma \text{ for all } k \in \N.
$

Fix 
$
c \in (0, \infty) \setminus \Gamma^{c}_*, \text{ and set } \beta_k = c \alpha_k \text{ for every } k \in \N.
$

\noindent \emph{{Third step.}} We construct the full-measure set $G\subset \Rm$ of good initialization points.

Define
$\displaystyle
E \coloneq\bigcap_{i=1}^{n}\left(\Omega_{f_i} \cap O_i\right),
$
where $\Omega_{f_i}$ is the set of points where $f_i$ is differentiable and its gradient $\nabla f_i$ equals its piecewise gradient $\widetilde{\nabla} f_i$. By \cite[Corollary 1]{despres2025associated},  $\Omega_{f_i}$ has full Lebesgue measure. Therefore,
$
\lambda^m(\Rm \setminus E)=0.
$ 
For every finite word
$\mathbf{i}=(i_0,\ldots,i_{\ell-1})\in\{1,\ldots, n\}^{\ell}$,
let
$
\Psi_{\mathbf{i}}^{(\ell)}
=
\psi_{i_{\ell -1},\beta_{\ell-1}}\circ\cdots\circ \psi_{i_0,\beta_0},
$
with the convention
$
\{1,\ldots,n\}^{0}=\{\varnothing\}
$
and
$
\Psi^{(0)}_{\varnothing}=\operatorname{Id}_m.
$
We now define the set of admissible initial points by
$$
G
\coloneq
\bigcap_{\ell=0}^{\infty}
\;
\bigcap_{\mathbf{i}\in\{1,\ldots,n\}^{\ell}}
\left(\Psi_{\mathbf{i}}^{(\ell)}\right)^{-1}(E).
$$
Then $x\in G$ if and only if
$
\Psi_{\mathbf{i}}^{(\ell)}(x)\in E
$
for every $\ell\geq 0$ and every finite word
$\mathbf{i}\in\{1,\ldots,n\}^{\ell}$.
Since
$\displaystyle
\Rm \setminus G
=
\bigcup_{\ell=0}^{\infty}
\;
\bigcup_{\mathbf{i}\in\{1,\ldots,n\}^{\ell}}
\left(\Psi_{\mathbf{i}}^{(\ell)}\right)^{-1}(\Rm \setminus E),
$
we want to prove that each set
$
\left(\Psi_{\mathbf{i}}^{(\ell)}\right)^{-1}(\Rm \setminus E)
$
has Lebesgue measure zero.
As the property~\eqref{eq:abscontdeter} is preserved under finite
composition, every map $\Psi_{\mathbf{i}}^{(\ell)}$ also satisfies
$\displaystyle
\lambda^m(A)=0
\implies
\lambda^m\!\left(
\left(\Psi_{\mathbf{i}}^{(\ell)}\right)^{-1}(A)
\right)=0
$
for any measurable set $A \subset \Rm$.
Applying this with $A=\Rm \setminus E$ gives
$
\lambda^m\!\left(
\left(\Psi_{\mathbf{i}}^{(\ell)}\right)^{-1}(\Rm \setminus E)
\right)=0.
$
For each fixed $\ell$, there are only $n^{\ell}$ words of length $\ell$,
so
$\displaystyle
\bigcup_{\mathbf{i}\in\{1,\ldots,n\}^{\ell}}
\left(\Psi_{\mathbf{i}}^{(\ell)}\right)^{-1}(\Rm \setminus E)
$
is a finite union of null sets and therefore has measure zero. Finally,
$
\Rm \setminus G
$
is a countable union of null sets. Hence,
$
\lambda^m(\Rm \setminus G)=0,
$
and $G$ has full Lebesgue measure.

\noindent \emph{{Fourth step.}} We prove convergence results for
\eqref{eq:gd}.

Let $x_0 \in G$. Then for all $k \in \N$, $\mathbb{P}$-a.s.,
\begin{align*}
    x_k \in E \quad \text{and}\quad x_{k+1} = x_k - c\alpha_k \nabla f_{i_k}(x_k) = x_k - c\alpha_k \nabla f(x_k) - c\alpha_k \eta_{k},
\end{align*}
where $\eta_k \coloneq \nabla f_{i_{k}}(x_{k}) - \nabla f(x_{k})$ for every $k \in \N$. 

We know that $\mathbb{P}$-a.s., for all $k \in \N$, $\mathbb{E}\left[{\nabla}f_{i_k}(x_k) \mid \mathfrak{F}_{k}\right] = {\nabla} f(x_k)$ and $\mathbb{E}[\eta_{k} \mid \mathfrak{F}_{k}] = 0$, where $\mathfrak{F}_k$ is the $\sigma$-algebra generated by $\{x_0, i_0, \ldots, i_{k-1}\}$. We also have $\mathbb{P}$-almost surely, $\nabla f(x_{k}) \in \partial^c f(x_k)$ for all $k \in \N$. 

The Clarke subdifferential of $f$ is a conservative field for $f$ by Theorem~\myref{theo:main}. It has nonempty, convex and compact values, and is graph-closed and locally bounded. The flow
\begin{equation}\label{eq:flow2}
    \dot{x}(t) \in - \partial^c f(x(t)),
\end{equation}
is well-defined and admits a solution by \cite[Remark~3.11]{bolte2023subgradient}. The function $\psi$ defined by 
$\displaystyle
\psi(x)\coloneq\max_{1\leq i\leq n}\max_{y \in \partial^{c}{f_i}(x)}\left\| y\right\|
$
is locally bounded. Recall that $\beta_k=c\alpha_k$. Since
$
\beta_k=o(1/\log k) \text{ as } k \to \infty
\text{ and }
\sum_{k=0}^{\infty}\beta_k=\infty,
$
the same argument used to establish \eqref{eq:condbenaim}, now applied
to $(\beta_k\eta_k)_{k\in\N}$, gives, for all $T>0$,
$$
\lim_{k\to\infty}
\sup_{N\geq k}
\left\{
\left\|
\sum_{\ell=k}^{N}\beta_{\ell}\eta_{\ell}
\right\|:
\sum_{\ell=k}^{N}\beta_{\ell}\leq T
\right\}
=0
\quad\mathbb{P}\text{-a.s.}
$$
$f$ is a Lyapunov function for the set of Clarke critical points $\{x\in\Rm:0\in\partial^c f(x)\}$ and for the flow~\eqref{eq:flow2} (see \cite[Lemma~3.13]{bolte2023subgradient}). Finally, Lemma~\myref{lem:morsesard} yields that the set of Clarke critical values $f\left(\{x\in\Rm:0\in\partial^c f(x)\}\right)$ has empty interior. All the hypotheses used in the proof of Theorem~\myref{theo:convconservative} are therefore satisfied with $\partial^c f$ in place of $\conv D_f$ and $\beta_k$ in place of $\alpha_k$. Repeating that argument gives the result.
\end{proof}
The following proposition shows that, without the assumption on the interface in \eqref{eq:interface}, a piecewise-$C^{\infty}$ function need not be locally $C^1$ on a set of positive measure and hence may fail to be locally $C^2$ a.e. as required in the proof of Theorem~\myref{theo:clarkcrit}.
\begin{proposition}\label{claim:counter}
For all $m \in \N\setminus\{0\}$, there exists a Lipschitz continuous and piecewise-$C^{\infty}$ function
$
f\colon\Rm\to\R
$
as in Definition~\myref{def:piecewisec1}
that is not locally $C^1$ on a set of positive measure.
\end{proposition}
\begin{proof}
We explicitly construct such a function.

Fix $m \in \N\setminus\{0\}$. Let $C\subset[0,1]$ be a fat Cantor set, also called a
Smith--Volterra--Cantor set. Then $C$ is compact, has empty interior, and has
positive Lebesgue measure.

From \cite[Lemma A.6]{bolte2026adjoint}, there exists a sequence $(z_{\ell})_{{\ell}\in \N} \subset \R \setminus C$ whose set of accumulation points is exactly $C$ and $z_i \neq z_{\ell}$ for $i \neq \ell$. For any ${\ell} \in \N$, consider the interval $I_{\ell} \coloneq (z_{\ell} - r_{\ell}, z_{\ell} + r_{\ell})$ with $0 < r_{\ell} < 1$ and $r_{\ell} \to 0$. We can choose the radii sufficiently small such that for all $\ell, i, s \in \N$, $\cl I_{\ell} \cap C = \varnothing$ and $\cl I_{i} \cap \cl I_s = \varnothing$ for $s \neq i$. Such a choice is always possible because each $z_{\ell}$ sits at a positive distance from both $C$ and all other points in the sequence.

Let $\varphi\colon \R \to \R$ be a $C^\infty$ bump function defined as
$$
\varphi(x)
=
\begin{cases}
e^{-1/(1 - x^2)} & \text{if } |x| < 1, \\
0 & \text{if } |x| \geq 1.
\end{cases}
$$
For all $\ell \in \N $, let $\varepsilon_{\ell} = e^{-1/r_{\!\ell}^{\ell}}$. Define the component functions $\phi_{\ell} \colon \R \to \R$ by
$$
\phi_{\ell}(x) \coloneq \varepsilon_{\ell}\left(x-z_{\ell}\right)\varphi\left(\frac{x-z_{\ell}}{r_{\ell}}\right),
$$
and define the aggregate function
$\displaystyle
g(x) \coloneq \sum_{{\ell}=0}^\infty \phi_{\ell}(x).
$

One can verify that, for all $i \in \N $, the series $\sum_{\ell=0}^{\infty}\phi_{\ell}^{(i)}$ of $i$-th derivatives converges uniformly. Thus, $g$ is $C^\infty$ on $\R$ by induction on $i$, and $g^{(i)}(x) = 0$ for all $x \in C$ and all $i \in \N$. Because the intervals $I_{\ell}$ are pairwise disjoint, evaluating $g$ and its derivative at the center points yields
$$
g(z_{\ell})=0 \quad \text{and} \quad g^{\prime}(z_{\ell}) = \varepsilon_{\ell}\varphi(0) > 0.
$$
Thus, every $z_{\ell}$ is a simple zero of $g$, so $g$ changes sign at each $z_{\ell}$. 

For $x=(x_1,x^{*})\in\R\times\R^{m-1}$, define
$$
f(x)\coloneq|g(x_1)|.
$$
When $m=1$, the variable $x^{*}$ is simply omitted. Let
$$
P^1\coloneq\{x\in\Rm:g(x_1)\geq0\},
\qquad
P^2\coloneq\{x\in\Rm:g(x_1)<0\},
$$
and define the globally $C^\infty$ functions
$$
f^1(x) \coloneq g(x_1),
\qquad
f^2(x) \coloneq -g(x_1).
$$
Then $\{P^1,P^2\}$ is a finite Borel partition of $\Rm$ and
$$
f=f^1\quad\text{on }P^1,
\qquad
f=f^2\quad\text{on }P^2.
$$
Hence, $f$ is piecewise-$C^\infty$ with respect to a finite Borel partition. Since the absolute value is $1$-Lipschitz and $g^{\prime}$ is globally bounded on $\R$, $f$ is Lipschitz continuous on $\Rm$.

Set
$\displaystyle
U \coloneq \bigcup_{\ell \in \N} (z_{\ell} - r_{\ell}, z_{\ell}).
$
Since $\varphi > 0$ on $(-1, 1)$,
$
g(x_1) < 0 \iff x_1 \in U.
$
Therefore
$$
P^1 = (\R \setminus U) \times \R^{m-1}, \quad P^2 = U \times \R^{m-1}.
$$
Because $r_{\ell} \to 0$ and the accumulation set of $(z_{\ell})$ is exactly $C$,
$\displaystyle
\cl {U} = C \cup \bigcup_{\ell \in \mathbb{N}} [z_{\ell} - r_{\ell}, z_{\ell}].
$
Thus the interface is
$$
\mathcal{I} = \cl P^1 \cap \cl P^2
= \left( (\R \setminus U) \cap \cl {U} \right) \times \R^{m-1}
= \left( C \cup \{ z_{\ell} - r_{\ell}, z_{\ell} : \ell \in \mathbb{N} \} \right) \times \R^{m-1},
$$
and has positive Lebesgue measure because $C$ is a fat Cantor set. Hence, the interface condition~\eqref{eq:interface} fails. 

We now show that $f$ is not locally $C^1$ on a set of positive measure. Fix $\ell\in\N$ and $y\in\R^{m-1}$, and set
$
x_{\ell,y}\coloneq(z_{\ell},y).
$
Since $g^{\prime}(z_{\ell})>0$, one has
$$
\lim_{h\to0^-}
\frac{f(x_{\ell,y}+he_1)-f(x_{\ell,y})}{h}
=
-g^{\prime}(z_{\ell}),
$$
whereas
$$
\lim_{h\to0^+}
\frac{f(x_{\ell,y}+he_1)-f(x_{\ell,y})}{h}
=
g^{\prime}(z_{\ell}),
$$
where $e_1 = (1, 0, \ldots, 0) \in \Rm$. Therefore, $f$ is not differentiable at any point $x_{\ell,y}$.

Now fix $c\in C$ and $y\in\R^{m-1}$. Since $c$ is an accumulation point of
$(z_{\ell})_{\ell\in\N}$, every neighborhood of $(c,y)$ contains a point of the
form $(z_{\ell},y)$ at which $f$ is not differentiable. 
It follows that $f$ is
not locally $C^1$ at $(c,y)$, and the set
$
C\times[0,1]^{m-1}
$
has positive Lebesgue measure and consists entirely of points
at which $f$ is not locally $C^1$.
\end{proof}

\section{Piecewise smoothness and Whitney stratification}\label{sect:counter}
In this section, we compare the class of piecewise smooth functions and that of functions whose graphs admit a Whitney stratification. 

Let us start by defining a Whitney stratification and Whitney stratifiable functions.
Let $G(q,\Rm)$ denote the Grassmannian of $q$-dimensional linear subspaces of $\Rm$, endowed with its natural manifold topology. Convergence of tangent spaces is understood with respect to this topology. 
\begin{definition}\label{def:strat}
Let $X\subset\Rm$ and let $I \subset \N$. A \emph{Whitney stratification} of $X$ is a locally finite
partition
$
\mathcal{S}=\{S_i\}_{i\in I}
$
of $X$ into pairwise disjoint $C^1$ submanifolds of $\Rm$, called
\emph{strata}, satisfying the following properties:

\begin{enumerate}
\item \textbf{Frontier condition.} For any two strata
$S,R\in\mathcal{S}$,
$$
S\cap\cl R\neq\varnothing
\implies
S \subset \cl R.
$$

\item \textbf{Whitney's condition $(a)$.} Whenever $S,R\in\mathcal{S}$ satisfy
$
S\subset \cl R,
$
and whenever $(x_{\ell})_{\ell\in\N}\subset R$ converges to a point $x\in S$ such that
the tangent spaces $T_{x_{\ell}}R$ converge to a linear subspace $L$ of $\Rm$, one
has
$
T_x S\subset L,
$
\end{enumerate}
where $T_{x} S$ (respectively, $T_{x_{\ell}}R$) denotes the tangent space of the manifold $S$ at $x$ (respectively, $R$ at $x_{\ell}$).
\end{definition}
\begin{remark}
    We do not assume that the strata are connected, although this stronger assumption appears in some definitions in the literature. We also note the existence of the stronger Whitney's condition (b), which will not be used here. Moreover, higher regularity may be imposed by requiring the strata to be of class $C^p$ for some $p \geq 1$. Unless otherwise specified, by Whitney stratification we mean a $C^1$ stratification satisfying Whitney's condition (a).
\end{remark}
\begin{definition}\label{def:whit} With a slight abuse of terminology, we call a function whose graph admits a Whitney stratification a \emph{Whitney stratifiable function}.
\end{definition}
\begin{remark}\label{rmk:counter}
    It is known that, in dimensions $m \geq 2$, the Euclidean norm is semialgebraic but is not piecewise-$C^p$, because the origin is its unique point of nondifferentiability, and piecewise-$C^p$ functions cannot have isolated points of nondifferentiability \cite[Theorem~1]{rockafellar2003property}.

    Conversely, since the classes overlap---the $\ell_1$ norm is a standard common example---we now ask whether every piecewise smooth function is Whitney stratifiable. From Proposition~\myref{claim:counter}, there exists a piecewise smooth function which is not locally $C^1$ on a set of positive measure, hence this function cannot be Whitney stratifiable. Indeed, for a locally Lipschitz and Whitney stratifiable function, a stratification of its graph induces a stratification of its domain, and the restriction of the function to each stratum is $C^1$; see \cite[Section~3]{bolte2007clarke}. The union of the lower-dimensional strata has measure zero, while the function is locally $C^1$ on the union of full-dimensional strata, since these strata are open.
\end{remark}
In the next theorem, we prove that, even with the interface condition \eqref{eq:interface} which, by Lemma~\myref{lem:piecewiseCplocallyC2ae}, ensures local smoothness of arbitrary order a.e., a piecewise smooth function is not necessarily Whitney stratifiable.
\begin{theorem}\label{theo:counter2}
For all $m \in \N\setminus\{0\}$, there exists a Lipschitz continuous and piecewise-$C^{\infty}$ function
$f\colon\Rm\to\R$ as in Definition~\myref{def:piecewisec1} with
Lebesgue-negligible interface,
\begin{equation*}
    \mathcal{I}
    \coloneq
    \bigcup_{\substack{j,\ell\in\{1,\ldots,J\}\\j\neq \ell}}
    \left(\cl{P^j}\cap\cl{P^{\ell}}\right),
\end{equation*}
whose graph admits no Whitney stratification.
\end{theorem}
\begin{proof} Fix $m \in \N\setminus\{0\}$. We construct a Lipschitz continuous and  piecewise-$C^{\infty}$ function $f\colon \Rm \to \R$ that satisfies the condition on the interface $\mathcal{I}$ and whose graph is not Whitney stratifiable.

Now let $C$ be the standard middle-thirds Cantor set. Thus, $C$ is compact, perfect, has empty interior, and has Lebesgue measure zero. Construct $f\colon \Rm \to \R$ by repeating the construction in the proof of Proposition~\myref{claim:counter} with this choice of $C$. We recall that $f$ is Lipschitz continuous and piecewise-$C^{\infty}$. For this representation of $f$, the interface is
$$
\mathcal{I}
=
\left(
C\cup\{z_{\ell}-r_{\ell},z_{\ell}:\ell\in\N\}
\right)\times\R^{m-1},
$$
and has Lebesgue measure zero. 

Now, suppose, towards a contradiction, that $\graph f$ admits a Whitney stratification $\mathcal{S}$.

We first determine the possible dimensions of the strata containing
nondifferentiability points of $f$. Fix $\ell\in\N$ and
$y\in\R^{m-1}$, and set
$$
w_{\ell,y}\coloneq(z_{\ell},y,0)\in\graph f,
\qquad
d_{\ell}\coloneq g^{\prime}(z_{\ell})>0.
$$
Let $S\in\mathcal{S}$ be the stratum containing $w_{\ell,y}$. We claim that
$
T_{w_{\ell,y}}S
\subset
\{0\}\times\R^{m-1}\times\{0\}.
$

Indeed, let
$
w=(u,v,\eta)\in T_{w_{\ell,y}}S.
$
There exists a $C^1$ curve
$
\gamma\colon(-1,1)\to S
$
such that
$
\gamma(0)=w_{\ell,y}
$
and
$
\dot{\gamma}(0)=w.
$
Since $\gamma$ takes values in $\graph f$, we may write
$$
\gamma(t)
=
\bigl(
z_{\ell}+\alpha(t),
y+\beta(t),
|g(z_{\ell}+\alpha(t))|
\bigr),
$$
where
$
\alpha(0)=0,
\dot{\alpha}(0)=u,
\beta(0)=0,
\dot{\beta}(0)=v.
$
Since $g(z_{\ell})=0$ and $g^{\prime}(z_{\ell})=d_{\ell}>0$, one has
$
|g(z_{\ell}+s)|
=
d_{\ell}|s|+o(|s|)
\text{ as }s \to 0.
$
If $u\neq0$, then
$
\alpha(t)=ut+o(t),
$
and consequently
$$
\lim_{t\to0^+}
\frac{|g(z_{\ell}+\alpha(t))|}{t}
=
d_{\ell}|u|,
\quad \text{and} \quad
\lim_{t\to0^-}
\frac{|g(z_{\ell}+\alpha(t))|}{t}
=
-d_{\ell}|u|.
$$
This contradicts the differentiability of the last component of $\gamma$ at
$t=0$. Therefore, $u=0$. It follows that $\alpha(t)=o(t)$, and hence
$
|g(z_{\ell}+\alpha(t))|
=
o(|t|).
$
Thus, $\eta=0$, which proves the claim:
$
T_{w_{\ell,y}}S
\subset
\{0\}\times\R^{m-1}\times\{0\}.
$
In particular,
$
\dim S \leq m-1.
$

We next show that every point of
$
C\times\R^{m-1}\times\{0\}
$
also belongs to a stratum of dimension at most $m-1$. 

Fix
$c\in C$ and $y\in\R^{m-1}$, and let
$
u \coloneq(c,y,0)\in\graph f.
$
Choose a subsequence $(z_{\ell_j})_{j\in\N}$ converging to $c$, and set
$
u_j\coloneq(z_{\ell_j},y,0).
$
Let $S\in\mathcal{S}$ be the stratum containing $u$. By local finiteness, there
exists a neighborhood $U$ of $u$ that intersects only finitely many strata.
Since $u_j\to u$, all sufficiently large $u_j$ belong to $U$. After passing to
a subsequence, there is therefore a single stratum $R\in\mathcal{S}$ such that
$
u_j\in R
\text{ for all }j\in\N.
$
The preceding argument gives
$
\dim R \leq m-1.
$
If $S=R$, then immediately $\dim S\leq m-1$. Suppose that $S\neq R$. Since
$u \in S\cap \cl R$, the frontier condition yields
$
S\subset \cl R.
$
The Grassmannian of $(\dim R)$-dimensional linear subspaces of $\R^{m+1}$ is
compact. Thus, after passing to a further subsequence, there exists a linear
subspace $L\subset\R^{m+1}$ such that
$
T_{u_j}R \to L.
$
Whitney's condition $(a)$, applied to the pair $(S,R)$, gives
$
T_uS\subset L.
$
Consequently,
$
\dim S
\leq
\dim L
=
\dim R
\leq
m-1.
$
Therefore, any point of
$
C\times\R^{m-1}\times\{0\}
$
belongs to a stratum of dimension at most $m-1$.

Set
$
K
\coloneq
C\times[0,1]^{m-1}\times\{0\}.
$
By local finiteness and compactness, $K$ intersects and is covered by only finitely many strata,
say $S_1,\ldots,S_N$. By the above argument,
$
\dim S_i\leq m-1
\text{ for every }i=1,\ldots,N.
$
It follows that the Hausdorff dimension of the union is at most $m - 1$, i.e.,
$\displaystyle
\dim_{\mathcal{H}}\left( \bigcup_{i=1}^N S_i\right) \leq m - 1.
$
Since $K \subset \bigcup_{i=1}^N S_i$, 
we have
\begin{align*}
    m-1 \geq \dim_{\mathcal{H}} K &\geq \dim_{\mathcal{H}} C + \dim_{\mathcal{H}} [0,1]^{m-1} + \dim_{\mathcal{H}} \{0\} \geq \frac{\log 2}{\log 3} + m - 1 > m - 1.
\end{align*}
This is a contradiction. Therefore, $\graph f$ admits no Whitney stratification.
\end{proof}

\runinsectionstar[\large]{Acknowledgements} The author warmly thanks J\'er\^ome Bolte and Edouard Pauwels for helpful discussions and for continued support and advice. The author also thanks Edouard Pauwels for suggesting the application section.

\runinsectionstar[\large]{Funding} This work has been supported by the Occitanie region, the European Regional Development Fund (ERDF), and the French government, through the France 2030 project managed by the National Research Agency (ANR) with the reference number ``ANR-22-EXES-0015''.

\bibliographystyle{plainnat}%
\bibliography{references_paper}
\end{document}